\documentclass[a4,11pt]{article}
\usepackage{graphicx}
\usepackage{amssymb}
\usepackage{url}
\usepackage{latexsym}
\usepackage{amssymb}
\usepackage{here}
\usepackage{wrapfig}
\usepackage{hyperref}
\usepackage{fancybox}
\usepackage{ascmac}
\usepackage{amsmath}
\usepackage{multicol}
\usepackage{mathtools}
\usepackage{amsthm}
\newtheorem{theo}{Theorem}[section]
\newtheorem{lem}{Lemma}[section]
\newtheorem{rem}{Remark}[section]
\newtheorem{dif}{Definition}[section]
\newtheorem{coro}{Corollary}[section]
\allowdisplaybreaks[4]

\begin{document}
\begin{center}
\noindent{\bf\Large Estimation of the Sobolev embedding constant on domains with minimally smooth boundary}\vspace{10pt}\\
{\normalsize Kazuaki Tanaka$^{1,*}$, Kouta Sekine$^{2}$, Makoto Mizuguchi$^{1}$, Shin'ichi Oishi$^{2,3}$}\vspace{5pt}\\
{\it\normalsize $^{1}$Graduate School of Fundamental Science and Engineering, Waseda University,\\
$^{2}$Faculty of Science and Engineering, Waseda University,\\
$^{3}$CREST, JST}
\end{center}
{\bf Abstract}. In this paper, we propose a method for estimating the Sobolev type embedding constant from $W^{1,q}(\Omega)$ to $L^{p}(\Omega)$ on a domain $\Omega\subset \mathbb{R}^{n}$~$(n=2,3,\cdots)$ with minimally smooth boundary, where $p\in(n/(n-1),\infty)$ and $q=np/(n+p)$.
We estimat the embedding constant by constructing an extension operator from $W^{1,q}(\Omega)$ to $W^{1,q}(\mathbb{R}^{n})$ and computing its operator norm.
We also present some examples of estimating the embedding constant for certain domains.\vspace{0.2cm}
\renewcommand{\thefootnote}{\fnsymbol{footnote}}
\footnote[0]{{\it E-mail addresse:} $^{*}$\href{mailto:imahazimari@fuji.waseda.jp}{\nolinkurl{imahazimari@fuji.waseda.jp}}\\[-3pt]}
\renewcommand\thefootnote{*\arabic{footnote}}

\section{Introduction}
Let $\Omega\subset \mathbb{R}^{n}~(n=2,3,\cdots)$ be a domain with minimally smooth boundary, whose definition will be introduced in Definition \ref{def/Lipschitz}.
We are concerned with a concrete value of the embedding constant $C_{p}(\Omega)$ from $W^{1,q}(\Omega)$ to $L^{p}(\Omega)$, i.e., $C_{p}(\Omega)$ satisfies
\begin{align}
\left\|u\right\|_{L^{p}(\Omega)}\leq C_{p}(\Omega)\left\|u\right\|_{W^{1,q}(\Omega)},~~\forall u\in W^{1,q}(\Omega),\label{purpose}
\end{align}
where $p\in(n/(n-1),\infty)$, $q=np/(n+p)$,
and the norm $\left\|\cdot\right\|_{W^{1,q}(\Omega)}$ denotes the $\sigma$-weighted $W^{1,q}$ norm defined as
\begin{align}
\left\|\cdot\right\|_{W^{1,q}(\Omega)}^{q}:=\left\|\nabla\cdot\right\|_{L^{q}(\Omega)}^{q}+\sigma\left\|\cdot\right\|_{L^{q}(\Omega)}^{q}\label{sigmanorm}
\end{align}
for given $\sigma>0$.

 Since the Sobolev type embedding theorems are important in studies on partial differential equations (PDEs),
there have been a lot of works on such theorems and their applications, e.g., \cite{aubin1976, aubin1982book, beckner1993, brezis1983, calderon1961, edmunds1987book, hestenes1941, lieb1983, nakao2001numerical, plum2001computer, plum2008, stein1970book, takayasu2013verified, talenti1976, watanabe2009symmetrization, whitney1934}.
In particular, a concrete value of the embedding constant is indispensable for verified numerical computation and compute-assisted proof for PDEs; see, e.g., \cite{nakao2001numerical, plum2001computer, plum2008, takayasu2013verified}.
We shall remark that the best constant in the classical Sobolev inequality on $\mathbb{R}^{n}$ was independently shown by Aubin \cite{aubin1976} and Talenti \cite{talenti1976} in 1976 (see Theorem \ref{talentitheo}).
Moreover, since all elements $u$ in $W_{0}^{k,q}(\Omega)$, the closure of $C_{0}^{\infty}(\Omega)$ commonly defined, can be regarded as elements of $W^{k,q}(\mathbb{R}^{n})$ by zero extension outside $\Omega$, the embedding constant satisfying (\ref{purpose}) with the restriction $u\in W_{0}^{k,q}(\Omega)$ can be also estimated for a general domain $\Omega\subset \mathbb{R}^{n}$ by calculating the classical embedding constant on $\mathbb{R}^{n}$.
Removing the restriction, however, such a simple extension cannot be constructed.
To estimate the embedding constant without the restriction, we construct a linear and bounded operator $E$ from $W^{1,q}(\Omega)$ to $W^{1,q}(\mathbb{R}^{n})$ such that $(Eu)(x)=u(x)$ for all $x\in\Omega$, which is called the extension operator from $W^{1,q}(\Omega)$ to $W^{1,q}(\mathbb{R}^{n})$.
We then estimate bounds for the operator norm $A_{q}\left(\Omega\right)$ of 
$E$ satisfying
\begin{align}
\left\|\nabla\left(Eu\right)\right\|_{L^{q}\left(\mathbb{R}^{n}\right)}\leq A_{q}\left(\Omega\right)\left(\left\|\nabla u\right\|_{L^{q}\left(\Omega\right)}+\sigma^{1/q}\left\|u\right\|_{L^{q}\left(\Omega\right)}\right),~~\forall u\in W^{1,q}(\Omega),\label{intro/aq}
\end{align}
which will lead bounds for the embedding constant.
There have been some construction methods for the extension operators.
For example, the reflection method originates from Whitney \cite{whitney1934} and Hestenes \cite{hestenes1941}, whose summary can be found in, e.g., \cite{adams2003book, brezis2011book}.
The Calder\'on extension theorem originates from \cite{calderon1961}, 
which is summarized in, e.g., \cite{adams2003book}.
Moreover, Stein \cite{stein1970book} has shown that extension operators can be constructed on domains with minimally smooth boundary.

 The main contribution of this paper is to propose a formula giving a concrete value of $A_{q}\left(\Omega\right)$ for the extension operator constructed by Stein's method.
Stein first constructed an extension operator on the special Lipschitz domain and then expanded this to that on domains with minimally smooth boundary.
In his method, the regularized distance plays an important role, which is a $C^{\infty}$ function approximating the distance from a given closed set $S\subset \mathbb{R}^{n}$ to any point in its complement $S^{c}$.
After the appearance of Stein's construction method, the regularized distance was generalized to a one-parameter family of smooth functions by Fraenkel \cite{fraenkel1979}.
We will construct extension operators using Stein's method with the generalized regularized distance to derive the embedding constant.
\section{Preparation}\label{pre/sec}
Through out this paper, the following notation is used:
\begin{itemize}
\setlength{\parskip}{0mm}
\setlength{\itemsep}{0mm}
\item$\mathbb{N}=\{1,2,3,\cdots\}$ and $\mathbb{N}_{0}=\{0,1,2,\cdots\}$;
\item$B\left(x,r\right)$ is the open ball whose center is $x$ and whose radius is $r\geq 0$;
\item for any point $x=\left(x_{1},x_{2},\cdots x_{n}\right)\in \mathbb{R}^{n}$, define $\left|x\right|:=(\left|x_{1}\right|^{2}+\left|x_{2}\right|^{2}+\cdots+\left|x_{n}\right|^{2})^{\frac{1}{2}}$;
\item for any set $S\subset \mathbb{R}^{n},\ S^{c}$ is its complementary set and $\overline{S}$ is its closure set;
\item for any set $S\subset \mathbb{R}^{n}$ and any $\varepsilon>0$, define $S^{\varepsilon}:=\{x\in \mathbb{R}^{n}\ :\ B\left(x,\varepsilon\right)\subset S\}$;
\item for any point $x\in \mathbb{R}^{n}$ and any set $S\subset \mathbb{R}^{n}$, define $\displaystyle \mathrm{dist}\,(x,S):=\inf\{\left|x-y\right|\ :\ y\in S\}$;
\item for any function $f,\ \mathrm{supp}$\,$f$ denotes the support of $f$;
\item for any function $f$ over $\mathbb{R},\ f'$ denotes the ordinary derivative of $f$;
\item for any function $f$ over $\mathbb{R}^{n}~\left(n=2,3,\cdots\right),\ \partial_{x_{i}}f$ denotes the partial derivative of $f$ with respect to the $i$-th component $x_{i}$ of $x$.
\end{itemize}
Let $L^{p}\left(\Omega\right)$~$(1\leq p<\infty)$ be the functional space of $p$-th power Lebesgue integrable functions over $\Omega$.
Let $W^{k,p}(\Omega)$~$(k\in \mathbb{N},\ 1\leq p<\infty)$ be the $k$-th order $L^{p}$ Sobolev space on $\Omega$; in particular, we denote $H^{k}\left(\Omega\right):=W^{k,2}(\Omega)$.

\begin{dif}[Mollifier]\label{mollifier}
A nonnegative function $\rho\in C^{\infty}\left(\mathbb{R}^{n}\right)$ is said to be a mollifier if
\begin{align*}
\rho\left(x\right)=0\ \mathrm{for}\ \left|x\right|\geq 1
~~{\rm and}~~
\displaystyle \int_{\mathbb{R}^{n}}\rho\left(x\right)dx=1.
\end{align*}
\end{dif}
For example, the function
\begin{align}
\rho\left(x\right):=
\left\{\begin{array}{l l}
c\exp\left(\frac{-1}{1-\left|x\right|^{2}}\right),&\left|x\right|<1,\\
0,&\left|x\right|\geq 1
\end{array}\right. \label{sp/mol}
\end{align}
becomes a mollifier, where $c$ is chosen so that $\displaystyle \int_{\mathbb{R}^{n}}\rho\left(x\right)dx=1$.

In the following lemma, existence of a $C^{\infty}$ function approximating Lipschitz continuous functions is guaranteed. 
\begin{lem}[L.E.~Fraenkel \cite{fraenkel1979}]\label{fra}
Let $f:\mathbb{R}^{n}\rightarrow \mathbb{R}$ be a function satisfying Lipschitz continuous\index{Lipschitz continuous} condition, i.e.,
\begin{align*}
\left|f\left(x\right)-f\left(y\right)\right|\leq M\left|x-y\right|,\ \ \forall x,y\in \mathbb{R}^{n}
\end{align*}
holds for some $M>0$.
Suppose that there is an open set $G\subset \mathbb{R}^{n}$, s.t., $f(x)>0$ for all $x\in G$.
Then, for given any $\varepsilon\in(0,1)$, there is a function $g\in C^{\infty}(G)$ such that, for all $x\in G$
\begin{align}
&\left(1+\varepsilon\right)^{-2}f\left(x\right)\leq g\left(x\right)\leq\left(1-\varepsilon\right)^{-2}f\left(x\right),
\shortintertext{and}
&\left|\frac{\partial^{\alpha}}{\partial x^{\alpha}}g\left(x\right)\right|\leq P_{\alpha}M^{\alpha}\left\{\varepsilon f\left(x\right)\right\}^{1-\left|\alpha\right|},~~\forall\alpha\in \mathbb{N}_{0}^{n}\ \mathrm{with}\ \left|\alpha\right|\geq 1.\label{frankel1}
\end{align}
Here, $P_{\alpha}$ is a constant depending only on $\alpha$.
\end{lem}
\begin{rem}
One of the concrete values of $P_{\alpha}$ can be derived as follows:
Let $\rho$ be the mollifier defined in $(\ref{sp/mol})$.
Let $\rho_{*}:\mathbb{R}\rightarrow \mathbb{R}$ be the function, s.t., $\rho_{*}\left(\left|x\right|\right)=\rho\left(x\right),\ x\in \mathbb{R}^{n}$.
The multi-index $\alpha$ is written as $\alpha=\beta+\gamma$ for $\beta,\gamma\in \mathbb{N}_{0}^{n}$ with $\left|\gamma\right|=1$.
Then, the inequality $(\ref{frankel1})$ holds for
\begin{align}
P_{\alpha}=\displaystyle \int_{\mathbb{R}^{n}}\frac{\left|\frac{\partial^{\beta}}{\partial x^{\beta}}\rho_{1}(y)\right|\left(1+\left|y\right|\right)^{\left|\beta\right|}}{\left(1-\left|y\right|\right)}dy,\label{fra/Palpha}
\end{align}
where $\rho_{1}(y):=\left(n-1\right)\rho_{*}\left(\left|y\right|\right)+\left|y\right|\rho_{*}'\left(\left|y\right|\right)$.
\end{rem}
By applying the above lemma to the distance function, the regularized distance for any closed set can be derived:
\begin{dif}[Regularized distance]\label{RD}
Let $S$ be a closed set in $\mathbb{R}^{n}$.
For given any $\xi\in(0,1)$, there exists a function 
$\mathrm{RD}_{S,\xi} \in C^{\infty}\left(S^{c}\right)$
such that, for all $x\in S^{c}$,
\begin{align}
&\left(1+\xi\right)^{-2}\mathrm{dist}\left(x,S\right)\leq \mathrm{RD}_{S,\xi}\left(x\right)\leq\left(1-\xi\right)^{-2}\mathrm{dist}\left(x,S\right)
\shortintertext{and}
&\left|\frac{\partial^{\alpha}}{\partial x^{\alpha}}\mathrm{RD}_{S,\xi}\left(x\right)\right|\leq P_{\alpha}\left(\xi\, \mathrm{dist}\left(x,S\right)\right)^{1-\left|\alpha\right|},~~\forall\alpha\in \mathbb{N}_{0}^{n}\ \mathrm{with}\ \left|\alpha\right|\geq 1.\label{Palpha}
\end{align}
The function $\mathrm{RD}_{S,\xi}$ is called regularized distance from $S$.
\end{dif}

Next, we introduce two types of open sets:
\begin{dif}[Special Lipschitz domain \cite{stein1970book}]\label{SpLip}
Let $\phi:\mathbb{R}^{n-1}\rightarrow \mathbb{R}$~$(n=2,3,\cdots)$ be a function satisfying the Lipschitz condition, i.e.,
\begin{align*}
\left|\phi\left(x\right)-\phi\left(y\right)\right|\leq M\left|x-y\right|,\ \forall x,y\in \mathbb{R}^{n-1}
\end{align*}
holds for some $M>0$.
Then, $\Omega$ is called a special Lipschitz domain if it is written as $\Omega:=\{\left(x',x_{n}\right)\in \mathbb{R}^{n}\ :\ x_{n}>\phi\left(x'\right)\}$ with $x'=(x_{1},x_{2},\cdots,x_{n-1})\in \mathbb{R}^{n-1}$.
\end{dif}
The positive number $M$ in Definition \ref{SpLip} is called the Lipschitz constant of $\Omega$.
Generalizing the special Lipschitz domain, the domain with minimally smooth boundary is defined as follows:
\begin{dif}[Domain with minimally smooth boundary \cite{stein1970book}]\label{def/Lipschitz}
An open set $\Omega\subset \mathbb{R}^{n}$~$(n=2,3,\cdots)$ is said to be a domain with minimally smooth boundary if there exist $\varepsilon>0,\ N\in \mathbb{N},\ M>0,$ and a sequence $\left\{U_{i}\right\}_{i\in \mathbb{N}}$ of open
subsets of $\mathbb{R}^{n}$ such that
\begin{enumerate}
\item for any $x\in\partial\Omega, B\left(x,\varepsilon\right)\subset U_{i}$ holds for some $i\in \mathbb{N}${\rm ;}
\item no point in $\mathbb{R}^{n}$ belongs to more than $N$ of the $U_{i}${\rm ;}
\item for any $i\in \mathbb{N}$, there exists a special Lipschitz domain $\Omega_{i}$, whose Lipschitz bound is not more than $M$, such that $U_{i}\cap\Omega=U_{i}\cap\Omega_{i}$.
\end{enumerate}
\end{dif}
The positive number $M$ in Definition \ref{def/Lipschitz} is called the Lipschitz constant of $\Omega$, and $N$ in Definition \ref{def/Lipschitz} is called the overlap number of $\Omega$.
To avoid confusion, $M$ and $N$ are sometimes denoted by $M_{\Omega}$ and $N_{\Omega}$, respectively.

\section{Construction of extension operator}\label{Stein/construction}
Here, we describe Stein's construction method for extension operators \cite{stein1970book}.
Stein first constructed an extension operator on a special Lipschitz domain.
He then expanded this to that on a domain with minimally smooth boundary.
\subsection{Extension operator on special Lipschitz domain}\label{subsec/sp}
Let $\Omega'\subset \mathbb{R}^{n}$~$(n=2,3,\cdots)$ be a special Lipschitz domain; namely,
$\Omega'$ is written as the form $\Omega':=\{\left(x',x_{n}\right)\in \mathbb{R}^{n}$~$:$~$x_{n}>\phi\left(x'\right)\},\ x'=(x_{1},x_{2},\cdots,x_{n-1})\in \mathbb{R}^{n-1}$ with a Lipschitz continuous function $\phi:\mathbb{R}^{n-1}\rightarrow \mathbb{R}$ whose Lipschitz constant is $M_{\Omega'}$.
For given $\xi>0$, let $\mathrm{RD}_{\Omega',\xi}$ be the regularized distance with the bound $P_{\alpha}$ as in (\ref{Palpha}).
Moreover, for given $\tau>0$, let us define $g_{\Omega',\tau,\xi}^{*}:=\left(1+\tau\right)C_{\Omega',\xi}\mathrm{RD}_{\Omega',\xi}$ with $C_{\Omega',\xi}:=\left(1+\xi\right)^{2}\sqrt{1+M_{\Omega'}^{2}}$.
Then, for any $k\in N_{0}$ and any $p\in[1,\infty)$, the operator $E_{\Omega',\tau,\xi}$ defined by
\begin{align}
&(E_{\Omega',\tau,\xi}u)\left(x',x_{n}\right)\nonumber\\
:=&\left\{\begin{array}{l l}
u\left(x',x_{n}\right) ,&\forall\left(x',x_{n}\right)\in\overline{\Omega'},\\
\displaystyle \int_{1}^{\infty}u\left(x',x_{n}+tg_{\Omega',\tau,\xi}^{*}\left(x',x_{n}\right)\right)\psi\left(t\right)dt ,&\forall\left(x',x_{n}\right)\in\left(\overline{\Omega'}\right)^{c}
\end{array}\right.\label{extension/sp}
\end{align}
becomes extension operator from $W^{k,p}\left(\Omega'\right)$ to $W^{k,p}\left(\mathbb{R}^{n}\right)$,
where $\psi:\mathbb{R}\rightarrow \mathbb{R}$ is a function satisfying the following property 
\begin{align}
\displaystyle \int_{1}^{\infty}\psi\left(t\right)dt=1,~~\displaystyle \int_{1}^{\infty}t^{m}\psi\left(t\right)dt=0,~~\forall m\in \mathbb{N}.\label{phi/prop}
\end{align}

Note that, since $\left(1+M_{\Omega'}^{2}\right)^{-1/2}\mathrm{dist}\left(x,\left(\overline{\Omega'}\right)^{c}\right)\geq\phi\left(x'\right)-x_{n}$ for all $\left(x',x_{n}\right)\in\left(\overline{\Omega'}\right)^{c}$, we have $g_{\Omega',\tau,\xi}^{*}\left(x',x_{n}\right)\geq(1+\tau)(\phi\left(x'\right)-x_{n})$.
\begin{rem}
The extension operator on a special Lipschitz domain presented here is a little general one.
That is to say, Stein set $\tau$ and $\xi$ in concrete values because he focused on just proving the existence of the extension operators in his original theory {\rm \cite{stein1970book}}.
However, the selections of $\tau$ and $\xi$ are important for accuracy of the corresponding embedding constant.
\end{rem}
\subsection{Extension operator on domain with minimally smooth boundary}
Let $\Omega$ be a domain with minimally smooth boundary.
Let $\left\{U_{i}\right\}_{i\in \mathbb{N}}$ be the sequence as in Definition \ref{def/Lipschitz}.
Let $\varepsilon$ be a positive number satisfying that $U_{i}^{\frac{3}{4}\varepsilon}$ are not empty for all $i\in \mathbb{N}$
and if $\mathrm{dist}\left(x,\partial\Omega\right)\leq\varepsilon/2$ then $x\in U_{i}^{\frac{1}{2}\varepsilon}$ holds for some $i\in \mathbb{N}$.
Let $\rho$ be a mollifier, and put $\rho_{\varepsilon}\left(x\right):=\varepsilon^{-n}\rho\left(x\varepsilon^{-1}\right)$.
Let $\chi_{i}$ be the characteristic function of $U_{i}^{\frac{3}{4}\varepsilon}$, and put $\lambda_{i}^{\varepsilon}\left(x\right):=(\chi_{i}*\rho_{\frac{1}{4}\varepsilon})\left(x\right)$.
Put
\begin{align*}
U_{0}&=\left\{x\in \mathbb{R}^{n}~:\ \mathrm{d}\mathrm{i}\mathrm{s}\mathrm{t}\left(x,\Omega\right)<\frac{1}{4}\varepsilon\right\},\\
U_{+}&=\left\{x\in \mathbb{R}^{n}~:\ \mathrm{d}\mathrm{i}\mathrm{s}\mathrm{t}\left(x,\partial\Omega\right)<\frac{3}{4}\varepsilon\right\},
\shortintertext{and}
U_{-}&=\left\{x\in\Omega~:\ \mathrm{d}\mathrm{i}\mathrm{s}\mathrm{t}\left(x,\partial\Omega\right)>\frac{1}{4}\varepsilon\right\}.
\end{align*}
Let $\chi_{0},\ \chi_{+}$, and $\chi_{-}$ be the corresponding characteristic functions of $U_{0},\ U_{+}$, and $U_{-}$, respectively.
Let $\lambda_{0}^{\varepsilon}:=\chi_{0}*\rho_{\frac{1}{4}\varepsilon},\ \lambda_{+}^{\varepsilon}:=\chi_{+}*\rho_{\frac{1}{4}\varepsilon}$, and $\lambda_{-}^{\varepsilon}:=\chi_{-}*\rho_{\frac{1}{4}\varepsilon}$.
Put
\begin{align*}
\displaystyle \Lambda_{+}^{\varepsilon}:=\lambda_{0}^{\varepsilon}\frac{\lambda_{+}^{\varepsilon}}{\lambda_{+}^{\varepsilon}+\lambda_{-}^{\varepsilon}}~~{\rm and}~~\displaystyle \Lambda_{-}^{\varepsilon}:=\lambda_{0}^{\varepsilon}\frac{\lambda_{-}^{\varepsilon}}{\lambda_{+}^{\varepsilon}+\lambda_{-}^{\varepsilon}}.
\end{align*}
To each $U_{i}$ there corresponds a special Lipschitz domain $\Omega_{i}$ as in Definition \ref{def/Lipschitz}.
Let $E_{\Omega_{i},\tau,\xi}^{i}$ be the extension operator for each $\Omega_{i}$ constructed by (\ref{extension/sp}).
For any $k\in N_{0}$ and any $p\in[1,\infty)$, the following operator $E_{\Omega,\tau,\xi,\varepsilon}$ becomes extension operator from $W^{k,p}\left(\Omega\right)$ to $W^{k,p}\left(\mathbb{R}^{n}\right)$:
\begin{align}
\left(E_{\Omega,\tau,\xi,\varepsilon}u\right)\left(x\right):=\Lambda_{+}^{\varepsilon}\left(x\right)\left(\frac{\displaystyle\sum_{i=1}^{\infty}\lambda_{i}^{\varepsilon}\left(x\right)E_{\Omega_{i},\tau,\xi}^{i}\left(\lambda_{i}^{\varepsilon}u\right)\left(x\right)}{\displaystyle\sum_{i=1}^{\infty}\lambda_{i}^{\varepsilon}\left(x\right)^{2}}\right)+\Lambda_{-}^{\varepsilon}\left(x\right)u\left(x\right)\label{gene/extension/siki}
\end{align}
for all $x\in \mathbb{R}^{n}.$

Here, one can observe that
\begin{itemize}
\setlength{\parskip}{0mm}
\setlength{\itemsep}{0mm}
\item$\mathrm{supp}$\,$\lambda_{i}^{\varepsilon}\subset U_{i}$, and $\lambda_{i}^{\varepsilon}\left(x\right)=1$ if $x\in U_{i}^{\frac{1}{2}\varepsilon}$;
\item if $x\in \mathrm{supp}\,\Lambda_{+}^{\varepsilon}$, then $\displaystyle \sum_{i\in \mathbb{N}}\lambda_{i}^{\varepsilon}\left(x\right)\geq 1$;
\item bounds of the derivatives of $\lambda_{i}^{\varepsilon}$ are independent on $i$ but depend only on the $L^{1}$ norm of the corresponding derivatives of $\rho_{\frac{1}{4}\varepsilon}$;
\item$\lambda_{0}^{\varepsilon}\left(x\right)=1$ if $x\in\overline{\Omega}$;
\item$\lambda_{+}^{\varepsilon}\left(x\right)=1$ if $\mathrm{dist}\left(x,\partial\Omega\right)\leq\varepsilon/2$;
\item$\lambda_{-}^{\varepsilon}\left(x\right)=1$ if $ x\in\Omega$ and $\mathrm{dist}\left(x,\partial\Omega\right)\geq\varepsilon/2$;
\item the supports of $\lambda_{0}^{\varepsilon},\ \lambda_{+}^{\varepsilon}$, and $\lambda_{-}^{\varepsilon}$ are contained in the $\varepsilon/2$-neighborhood of $\Omega$, in the $\varepsilon$-neighborhood of $\partial\Omega$, and in $\Omega$, respectively;
\item the functions $\lambda_{0}^{\varepsilon},\ \lambda_{+}^{\varepsilon}$, and $\lambda_{-}^{\varepsilon}$ are bounded in $\mathbb{R}^{n}$, and all their partial derivatives are also bounded;
\item all the derivatives of $\Lambda_{+}^{\varepsilon}$ and $\Lambda_{-}^{\varepsilon}$ are bounded on $\mathbb{R}^{n}$;
\item$\Lambda_{+}^{\varepsilon}+\Lambda_{-}^{\varepsilon}$ is $1$ on $\overline{\Omega}$ and is $0$ outside the $\varepsilon/2$-neighborhood of $\Omega$.
\end{itemize}

\begin{rem}
In Stein's original method {\rm \cite{stein1970book}}, the assumption for $\varepsilon>0$ is just to be small enough.
However, since bounds for the derivatives of $\lambda_{i}^{\varepsilon}$ increase with decreasing $\varepsilon$,
a small $\varepsilon$ makes the corresponding extension constant large.
Due to this, we should select the value of $\varepsilon$ with taking this property in consideration.
The selections of $\varepsilon$ for concrete domains $\Omega$ can be seen in Subsection \ref{example}.
\end{rem}

\section{Formula for estimating operator norm}\label{main/theo/sec}
Let us first present the following lemma, which gives bounds for the operator norm of the extension operator on special Lipschitz domains constructed by the method in Subsection \ref{subsec/sp}.
\begin{lem}\label{sp/lemma}
For a special Lipschitz domain $\Omega'\subset \mathbb{R}^{n}\ (n=2,3,\cdots)$,
let $E\,(=E_{\Omega',\tau,\xi})$ be the extension operator constructed by $(\ref{extension/sp})$.
Then,
\begin{align}
\left\|Eu\right\|_{L^{p}\left(\mathbb{R}^{n}\right)}&\leq A_{p,\tau,\xi}\left(\Omega'\right)\left\|u\right\|_{L^{p}\left(\Omega'\right)},~~\forall u\in H^{1}(\Omega')\label{splip1}
\shortintertext{and}
\left\|\nabla\left(Eu\right)\right\|_{L^{p}\left(\mathbb{R}^{n}\right)}&\leq A_{p,\tau,\xi}'\left(\Omega'\right)\left\|\nabla u\right\|_{L^{p}\left(\Omega'\right)},~~\forall u\in H^{1}(\Omega')\label{splip2}
\end{align}
hold for
\begin{align*}
&A_{p,\tau,\xi}\left(\Omega'\right)=\left\{\left(A_{0}Q\right)^{p}+1\right\}^{1/p}
\shortintertext{and}
&A_{p,\tau,\xi}'\left(\Omega'\right)\\
=&\displaystyle \max\left\{2^{p-1}\left(A_{0}Q\right)^{p}+1,\left[\left(n-1\right)2^{p-1}\left(BQ\right)^{p}+\left\{\left(A_{0}+B\right)Q\right\}^{p}+1\right]^{1/p}\right\},
\end{align*}
respectively.
Here:
\begin{itemize}
\renewcommand{\labelitemi}{\normalfont\bfseries \textendash}
\setlength{\parskip}{0mm}
\setlength{\itemsep}{0mm}
\item $A_{0}$ and $A_{1}$ are constants satisfying $\left|\psi\left(t\right)\right|\leq A_{0}/t^{2}~(t\geq 1)$ and $\left|\psi\left(t\right)\right|\leq A_{1}/t^{3}~(t\geq 1)$, respectively;
\item $P$ is corresponding to $P_{\alpha}$ with $|\alpha|=1$.
\item $Q$ and $B$ are defined as
\begin{align*}
&Q~(=Q_{\Omega',\tau,\xi,p}):=\displaystyle \frac{p\left(1+\tau\right)\left(1+\xi\right)^{2}}{(p+1)\tau^{1+1/p}\left(1-\xi\right)^{2}}\sqrt{1+M_{\Omega'}^{2}}
\shortintertext{and}
&B~(=B_{\Omega',\xi,\tau}):=A_{1}P\left(1+\xi\right)^{2}\left(1+\tau\right)\sqrt{1+M_{\Omega'}^{2}}.
\end{align*}
\end{itemize}
\end{lem}
\begin{proof}
Since $C^{\infty}(\Omega')$ is dense in $H^{1}(\Omega')$, it suffices to consider $u\in C^{\infty}(\Omega')$.
Moreover, $\Omega'$ is written as the form $\Omega':=\{\left(x',x_{n}\right)\in \mathbb{R}^{n}$~$:$~$x_{n}>\phi\left(x'\right)\},\ x'=(x_{1},x_{2},\cdots,x_{n-1})\in \mathbb{R}^{n-1}$ with a Lipschitz continuous function $\phi:\mathbb{R}^{n-1}\rightarrow \mathbb{R}$ whose Lipschitz constant is $M_{\Omega'}$.
Hereafter, we write $u_{y}=\partial_{y}u,\ g^{*}=g_{\Omega',\tau,\xi}^{*},\ g_{y}^{*}=\partial_{y}g^{*},$ for simplicity.\\
\\
{\it The first step: estimating $A_{p,\tau,\xi}\left(\Omega'\right)$}\\[5pt]
If $y<\phi\left(x\right)$ with $y\in \mathbb{R}$ and $x=(x_{1},x_{2},\cdots,x_{n-1})\in \mathbb{R}^{n-1}$,
\begin{align}
\left|\left(Eu\right)\left(x,y\right)\right|&=\left|\int_{1}^{\infty}u\left(x,y+tg^{*}\left(x,y\right)\right)\psi\left(t\right)dt\right|\nonumber\\
&\displaystyle \leq A_{0}\int_{1}^{\infty}\left|u\left(x,y+tg^{*}\left(x,y\right)\right)\right|\frac{dt}{t^{2}}.\label{eq/a01}
\end{align}
Setting $z:=y-\phi\left(x\right)$, we have $g^{*}\left(x,y\right)\geq\left(1+\tau\right)\left(\phi\left(x\right)-y\right)=\left(1+\tau\right)\left|z\right|$.
We also have $\phi\left(x\right)-y\geq \mathrm{dist}\left(\left(x,y\right),\ \overline{\Omega'}\right)$ for all $x\in \mathbb{R}^{n-1}$ and $y\in \mathbb{R}$.
Since $\mathrm{dist}\left(\left(x,y\right),\ \overline{\Omega'}\right)\geq\left(1-\xi\right)^{2}\mathrm{RD}_{\Omega',\xi}\left(x_{,}y\right)$ holds, it follows that
\begin{align}
\left|z\right|=\phi\left(x\right)-y&\geq \mathrm{dist}\left(\left(x,y\right),\ \overline{\Omega'}\right)\nonumber\\
&\geq\left(1-\xi\right)^{2}\mathrm{RD}_{\Omega',\xi}\left(x,y\right)\nonumber\\
&=\left(1-\xi\right)^{2}\left(1+\tau\right)^{-1}C_{\Omega',\xi}^{-1}g^{*}\left(x,y\right).\label{omoidasi}
\end{align}
Now, recall that $g^{*}=\left(1+\tau\right)C_{\Omega',\xi}\mathrm{RD}_{\Omega',\xi}$.
From (\ref{omoidasi}), we obtain $g^{*}\left(x,y\right)\leq a\left|z\right|$, where $a~(=a_{\Omega',\tau,\xi}):=(1+\tau)(1+\xi)^{2}(1-\xi)^{-2}\sqrt{1+M_{\Omega'}^{2}}$.
Putting $s=z+tg^{*}\left(x,y\right)$, it follows from (\ref{eq/a01}) that
\begin{align*}
\left|\left(Eu\right)\left(x,y\right)\right|&\displaystyle \leq A_{0}\int_{1}^{\infty}\left|u\left(x,y+tg^{*}\left(x,y\right)\right)\right|\frac{dt}{t^{2}}\\
&=A_{0}g^{*}\displaystyle \left(x,y\right)\int_{z+g^{*}\left(x,y\right)}^{\infty}\left|u\left(x,s+\phi(x)\right)\right|\left(s-z\right)^{-2}ds\\
&\displaystyle \leq A_{0}a\left|z\right|\int_{\tau\left|z\right|}^{\infty}\left|u\left(x,s+\phi(x)\right)\right|\left(s-z\right)^{-2}ds\\
&\displaystyle \leq A_{0}a\left|z\right|\int_{\tau\left|z\right|}^{\infty}\left|u\left(x,s+\phi(x)\right)\right|s^{-2}ds.
\end{align*}
By changing the variable integration as $\left(\tau(y-\phi\left(x\right))=\right)\tau z=w$, we have
\begin{align*}
&\displaystyle \int_{-\infty}^{\phi(x)}\left|\left(Eu\right)\left(x,y\right)\right|^{p}dy\nonumber\\
\leq&\displaystyle \left(\frac{aA_{0}}{\tau}\right)^{p}\int_{-\infty}^{\phi(x)}\left(\tau\left|z\right|\int_{\tau\left|z\right|}^{\infty}\left|u\left(x,s+\phi(x)\right)\right|s^{-2}ds\right)^{p}dy,\ z=y-\phi\left(x\right)\nonumber\\
=&\displaystyle \left(\frac{aA_{0}}{\tau^{1+1/p}}\right)^{p}\int_{-\infty}^{0}\left(\left|w\right|\int_{\left|w\right|}^{\infty}\left|u\left(x,s+\phi(x)\right)\right|s^{-2}ds\right)^{p}dw\nonumber\\
=&\displaystyle \left(\frac{aA_{0}}{\tau^{1+1/p}}\right)^{p}\int_{0}^{\infty}\left(\int_{\left|w\right|}^{\infty}\left|u\left(x,s+\phi(x)\right)\right|s^{-2}ds\right)^{p}\left|w\right|^{(p+1)-1}dw.
\end{align*}
Hardy's inequality, which can be found in Lemma \ref{hardy}, gives
\begin{align}
\displaystyle \int_{-\infty}^{\phi(x)}\left|\left(Eu\right)\left(x,y\right)\right|^{p}dy&\displaystyle \leq\left(\frac{pA_{0}a}{(p+1)\tau^{1+1/p}}\right)^{p}\int_{0}^{\infty}\left(\left|u\left(x,s+\phi(x)\right)\right|s^{-1}\right)^{p}s^{p}ds\nonumber\\
&=\displaystyle \left(\frac{pA_{0}a}{(p+1)\tau^{1+1/p}}\right)^{p}\int_{0}^{\infty}\left|u\left(x,s+\phi(x)\right)\right|^{p}ds\nonumber\\
&=\displaystyle \left(\frac{pA_{0}a}{(p+1)\tau^{1+1/p}}\right)^{p}\int_{\phi(x)}^{\infty}\left|u\left(x,y\right)\right|^{p}dy.\label{tasumae}
\end{align}
Moreover, from the definition (\ref{extension/sp}) of the extension operator, we have
\begin{align}
\displaystyle \int_{\phi\left(x\right)}^{\infty}\left|\left(Eu\right)\left(x,y\right)\right|^{p}dy=\int_{\phi\left(x\right)}^{\infty}\left|u\left(x,y\right)\right|^{p}dy.\label{tasumono}
\end{align}
From (\ref{tasumae}) and (\ref{tasumono}), it follows that
\begin{align}
&\left(\int_{-\infty}^{\infty}\left|\left(Eu\right)\left(x,y\right)\right|^{p}dy\right)^{1/p}\nonumber\\
\leq&\left\{\left(A_{0}Q\right)^{p}+1\right\}^{1/p}\left(\int_{\phi\left(x\right)}^{\infty}\left|u\left(x,y\right)\right|^{p}dy\right)^{1/p},\label{pre/ap1}
\end{align}
where $Q~(=Q_{\Omega',\tau,\xi,p}):=pa_{\Omega',\tau,\xi}/\left\{(p+1)\tau^{1+1/p}\right\}$.
Integrating the both side of (\ref{pre/ap1}) by $x$, we find that (\ref{splip1}) holds for
\begin{align*}
A_{p,\tau,\xi}\left(\Omega'\right)=\left\{\left(A_{0}Q\right)^{p}+1\right\}^{1/p}.
\end{align*}
\\
{\it The second step: estimating $A_{p,\tau,\xi}^{'}\left(\Omega'\right)$}\\[5pt]
The inequality (\ref{Palpha}) ensures that $|g_{x_{j}}^{*}\left(x,y\right)|\leq B/A_{1}$ for $j\in\{1,2,\cdots,n\}$.
If $y<\phi\left(x\right)$ with $y\in \mathbb{R}$ and $x=(x_{1},x_{2},\cdots,x_{n-1})\in \mathbb{R}^{n-1}$,
\begin{align}
\partial_{y}\left(Eu\right)\left(x,y\right)&=\displaystyle \partial_{y}\int_{1}^{\infty}u\left(x,y+tg^{*}\ \left(x,y\right)\right)\psi\left(t\right)dt\nonumber\\
&=\displaystyle \int_{1}^{\infty}u_{y}\left(x,y+tg^{*}\ \left(x,y\right)\right)\left(1+tg_{y}^{*}\left(x,y\right)\right)\psi\left(t\right)dt\nonumber\\
&=\displaystyle \int_{1}^{\infty}u_{y}\left(x,y+tg^{*}\ \left(x,y\right)\right)\psi\left(t\right)dt\nonumber\\
&~~~~~~~~~~~~~~~+g_{y}^{*}\displaystyle \left(x,y\right)\int_{1}^{\infty}u_{y}\left(x,y+tg^{*}\left(x,y\right)\right)t\psi\left(t\right)dt.\nonumber
\end{align}
Therefore, we have
\begin{align}
&\left|\partial_{y}\left(Eu\right)\left(x,y\right)\right|\nonumber\\
\leq&\left|\int_{1}^{\infty}u_{y}\left(x,y+tg^{*}\left(x,y\right)\right)\psi\left(t\right)dt\right|\nonumber\\
&~~~~~~~~~~~~~~+\left|g_{y}^{*}\left(x,y\right)\right|\left|\int_{1}^{\infty}u_{y}\left(x,y+tg^{*}\left(x,y\right)\right)t^{3}\psi\left(t\right)\frac{dt}{t^{2}}\right|\nonumber\\
\leq&\displaystyle \left(A_{0}+B\right)\int_{1}^{\infty}\left|u_{y}\left(x,y+tg^{*}\left(x,y\right)\right)\right|\frac{dt}{t^{2}},~~y<\phi\left(x\right).\nonumber
\end{align}
From the similar discussion in the first step, we have
\begin{align}
\displaystyle \int_{-\infty}^{\infty}\left|\partial_{y}\left(Eu\right)\left(x,y\right)\right|^{p}dy\leq\left[\left\{\left(A_{0}+B\right)Q\right\}^{p}+1\right]\int_{\phi(x)}^{\infty}\left|u_{y}\left(x,y\right)\right|^{p}dy.\label{Adash/pre1}
\end{align}
On the other hand, for $j\in\left\{1,2,\cdots,n-1\right\}$ and $y<\phi\left(x\right)$,
\begin{align*}
&\partial_{x_{j}}\left(Eu\right)\left(x,y\right)\nonumber\\
=&~\displaystyle \partial_{x_{j}}\int_{1}^{\infty}u\left(x,y+tg^{*}\left(x,y\right)\right)\psi\left(t\right)dt\nonumber\\
=&\displaystyle \int_{1}^{\infty}\left\{u_{x_{j}}\left(x,y+tg^{*}\left(x,y\right)\right)+u_{y}\left(x,y+tg^{*}\left(x,y\right)\right)tg_{x_{j}}^{*}\left(x,y\right)\right\}\psi\left(t\right)dt\nonumber\\
=&\displaystyle \int_{1}^{\infty}u_{x_{j}}\left(x,y+tg^{*}\left(x,y\right)\right)\psi\left(t\right)dt\nonumber\\
&~~~~~~~~~~~~~~~~~~~+g_{x_{j}}^{*}\displaystyle \left(x,y\right)\int_{1}^{\infty}u_{y}\left(x,y+tg^{*}\left(x,y\right)\right)t\psi\left(t\right)dt.
\end{align*}
Therefore, we have
\begin{align}
&\left|\partial_{x_{j}}\left(Eu\right)\left(x,y\right)\right|\nonumber\\
\leq&\left|\int_{1}^{\infty}u_{x_{j}}\left(x,y+tg^{*}\left(x,y\right)\right)\psi\left(t\right)dt\right|\nonumber\\
&~~~~~~~~~~~~~~+\left|g_{x_{j}}^{*}\left(x,y\right)\right|\left|\int_{1}^{\infty}u_{y}\left(x,y+tg^{*}\left(x,y\right)\right)t\psi\left(t\right)dt\right|\nonumber\\
\leq&A_{0}\displaystyle \int_{1}^{\infty}\left|u_{x_{j}}\left(x,y+tg^{*}\left(x,y\right)\right)\right|\frac{dt}{t^{2}}\nonumber\\
&~~~~~~~~~~~~~~+B\displaystyle \int_{1}^{\infty}\left|u_{y}\left(x,y+tg^{*}\left(x,y\right)\right)\right|\frac{dt}{t^{2}},~~y<\phi\left(x\right)\nonumber.
\end{align}
Since $\left(s+t\right)^{p}\leq 2^{p-1}\left(s^{p}+t^{p}\right)$ holds for $s,t>0$ and $p>1$, it follows from the similar discussion in (\ref{tasumae}) that
\begin{align*}
&\displaystyle \int_{-\infty}^{\phi(x)}\left|\partial_{x_{j}}\left(Eu\right)\left(x,y\right)\right|^{p}dy\\
\leq&2^{p-1}\displaystyle \int_{-\infty}^{\phi(x)}\left|A_{0}a\left|z\right|\int_{\tau\left|z\right|}^{\infty}\left|u_{x_{j}}\left(x,s+\phi(x)\right)\right|s^{-2}ds\right|^{p}dy\\
&~~~~~~~~~~~~~~~+2^{p-1}\displaystyle \int_{-\infty}^{\phi(x)}\left|Ba\left|z\right|\int_{\tau\left|z\right|}^{\infty}\left|u_{y}\left(x,s+\phi(x)\right)\right|s^{-2}ds\right|^{p}dy\\
\leq&2^{p-1}\displaystyle \left(A_{0}Q\right)^{p}\int_{\phi(x)}^{\infty}\left|u_{x_{j}}\left(x,y\right)\right|^{p}dy+2^{p-1}\left(BQ\right)^{p}\int_{\phi(x)}^{\infty}\left|u_{y}\left(x,y\right)\right|^{p}dy.
\end{align*}
Therefore,
\begin{align}
&\displaystyle \int_{-\infty}^{\infty}\left|\partial_{x_{j}}\left(Eu\right)\left(x,y\right)\right|^{p}dy
\leq\displaystyle \left\{2^{p-1}\left(A_{0}Q\right)^{p}+1\right\}\int_{\phi(x)}^{\infty}\left|u_{x_{j}}\left(x,y\right)\right|^{p}dy\nonumber\\
&~~~~~~~~~~~~~~~~~~~~~~~~~~~~~~~~~~~~~~~~~~~~+2^{p-1}\left(BQ\right)^{p}\int_{\phi(x)}^{\infty}\left|u_{y}\left(x,y\right)\right|^{p}dy\label{Adash/pre2}
\end{align}
for $j\in\left\{1,2,\cdots,n-1\right\}$.
From (\ref{Adash/pre1}) and (\ref{Adash/pre2}), we have
\begin{align*}
&\displaystyle \sum_{j=1}^{n}\int_{-\infty}^{\infty}\left|\partial_{x_{j}}\left(Eu\right)\left(x,y\right)\right|^{p}dy\\
=&\displaystyle \sum_{j=1}^{n-1}\int_{-\infty}^{\infty}\left|\partial_{x_{j}}\left(Eu\right)\left(x,y\right)\right|^{p}dy+\int_{-\infty}^{\infty}\left|\partial_{y}\left(Eu\right)\left(x,y\right)\right|^{p}dy\\
\leq&\displaystyle \left\{2^{p-1}\left(A_{0}Q\right)^{p}+1\right\}\sum_{j=1}^{n-1}\int_{\phi(x)}^{\infty}\left|u_{x_{j}}\left(x,y\right)\right|^{p}dy\\
&~~~~~~~~~~~~~~~+\displaystyle \left(n-1\right)2^{p-1}\left(BQ\right)^{p}\int_{\phi(x)}^{\infty}\left|u_{y}\left(x,y\right)\right|^{p}dy\\
&~~~~~~~~~~~~~~~+\displaystyle \left[\left\{\left(A_{0}+B\right)Q\right\}^{p}+1\right]\int_{\phi(x)}^{\infty}\left|u_{y}\left(x,y\right)\right|^{p}dy\\
=&\displaystyle \left\{2^{p-1}\left(A_{0}Q\right)^{p}+1\right\}\sum_{j=1}^{n-1}\int_{\phi(x)}^{\infty}\left|u_{x_{j}}\left(x,y\right)\right|^{p}dy\\
&~~~~~~~~+\displaystyle \left[\left(n-1\right)2^{p-1}\left(BQ\right)^{p}+\left\{\left(A_{0}+B\right)Q\right\}^{p}+1\right]\int_{\phi(x)}^{\infty}\left|u_{y}\left(x,y\right)\right|^{p}dy.
\end{align*}
This ensures that the inequality (\ref{splip2}) holds for
\begin{align*}
&A_{p,\tau,\xi}'\left(\Omega'\right)\nonumber\\
=&\displaystyle \max\left\{2^{p-1}\left(A_{0}Q\right)^{p}+1,\left[\left(n-1\right)2^{p-1}\left(BQ\right)^{p}+\left\{\left(A_{0}+B\right)Q\right\}^{p}+1\right]^{1/p}\right\}.
\end{align*}
\end{proof}
The following formula enable us to estimate the operator norm $A_{q}\left(\Omega\right)$ for the extension operator on domains with minimally smooth boundary constructed by the method in Section \ref{Stein/construction}.
\begin{theo}\label{main}
For a domain $\Omega\subset \mathbb{R}^{n}$ $(n=2,3,\cdots)$ with minimally smooth boundary,
let $E\,(=E_{\Omega,\tau,\xi,\varepsilon})$ be the extension operator constructed by $(\ref{gene/extension/siki})$.
Then, letting $\gamma$ be a given positive number,
\begin{align}
\left\|\nabla\left(Eu\right)\right\|_{L^{p}\left(\mathbb{R}^{n}\right)}\leq A_{p}\left(\Omega\right)\left(\left\|\nabla u\right\|_{L^{p}\left(\Omega\right)}+\gamma\left\|u\right\|_{L^{p}\left(\Omega\right)}\right),~~\forall u\in W^{1,p}(\Omega)\label{ap/condision}
\end{align}
holds for
\begin{align}
A_{p}\left(\Omega\right)=\left\{\begin{array}{l}
NA'+1,~~R\leq\gamma,\\
b_{\varepsilon}\left(6NA+NA'+3\right)n^{1/p}/\gamma,~~R>\gamma,
\end{array}\right.\label{main/ap}
\end{align}
where $N$ is the overlap number of $\Omega$,
$b_{\varepsilon}$ is a positive number satisfying $b_{\varepsilon}\geq\int_{\mathbb{R}^{n}}|\partial_{x_{j}}\rho_{\frac{1}{4}\varepsilon}\left(x\right)|dx$ for all $j\in\{1,2,\cdots,n\}$,
and $R:=b_{\varepsilon}(6NA+NA'+3)n^{1/p}/(NA'+1)$.
The constants $A$ and $A'$ are determined by $A=\sup\{A_{p,\tau,\xi}\left(\Omega_{i}\right) : i\in\mathbb{N}\}$ and $A'=\sup\{A'_{p,\tau,\xi}\left(\Omega_{i}\right) : i\in\mathbb{N}\}$ for the operator norms $A_{p,\tau,\xi}\left(\Omega_{i}\right)$ and $A'_{p,\tau,\xi}\left(\Omega_{i}\right)$ of $E^{i}\,(=E_{\Omega_{i},\tau,\xi}^{i})$ satisfying $(\ref{splip1})$ and $(\ref{splip2})$ with the notational replacement $\Omega'=\Omega_{i}$, respectively.
\end{theo}
\begin{proof}
For any $j\in\{1,2,\cdots,n\}$, we have
\begin{align}
\displaystyle \left|\partial_{x_{j}}\lambda_{i}^{\varepsilon}\right|\leq\int_{\mathbb{R}^{n}}\left|\partial_{x_{j}}\rho_{\frac{1}{4}\varepsilon}\left(x\right)\right|dx\leq b_{\varepsilon};\label{bepsilon}
\end{align}
this bound does not depend on the index $i$.
Likewise, $\left|\partial_{x_{j}}\lambda_{0}^{\varepsilon}\right|,\ \left|\partial_{x_{j}}\lambda_{+}^{\varepsilon}\right|$, and $\left|\partial_{x_{j}}\lambda_{-}^{\varepsilon}\right|$ are bounded by $b_{\varepsilon}$.
Moreover,
\begin{align*}
\left|\partial_{x_{j}}\Lambda_{+}^{\varepsilon}\right|&=\left|\left(\partial_{x_{j}}\lambda_{0}^{\varepsilon}\right)\frac{\lambda_{+}^{\varepsilon}}{\lambda_{+}^{\varepsilon}+\lambda_{-}^{\varepsilon}}+\lambda_{0}^{\varepsilon}\frac{\left(\partial_{x_{j}}\lambda_{+}^{\varepsilon}\right)\left(\lambda_{+}^{\varepsilon}+\lambda_{-}^{\varepsilon}\right)-\lambda_{+}^{\varepsilon}\left(\partial_{x_{j}}\lambda_{+}^{\varepsilon}+\partial_{x_{j}}\lambda_{-}^{\varepsilon}\right)}{\left(\lambda_{+}^{\varepsilon}+\lambda_{-}^{\varepsilon}\right)^{2}}\right|\nonumber\\
&\leq 3b_{\varepsilon}=:b_{+}.
\end{align*}
It is easily confirmed that $\left|\partial_{x_{j}}\Lambda_{-}^{\varepsilon}\right|$ is also bounded by $3b_{\varepsilon}=:b_{-}$ (we distinguish $b_{+}$ and $b_{-}$ to avoid confusion in the following proof).
Hereafter, we simply denote
$\displaystyle \bigcup_{i}U_{i}^{\varepsilon/2}$ by $U^{*}$,
$\displaystyle \sum_{i\in \mathbb{N}}$ by $\displaystyle \sum$,
$\lambda_{i}^{\varepsilon}$ by $\lambda_{i}$,
$\Lambda_{+}^{\varepsilon}$ by $\Lambda_{+}$,
and $\Lambda_{-}^{\varepsilon}$ by $\Lambda_{-}$.
For $u\in H^{1}(\Omega)$,
\begin{align}
&\left\|\nabla\left(Eu\right)\right\|_{L^{p}\left(\mathbb{R}^{n}\right)}\nonumber\\
=&\left(\sum_{j}\int_{\mathbb{R}^{n}}\left|\partial_{x_{j}}\left(Eu\right)\right|^{p}dx\right)^{1/p}\nonumber\\
\leq&\left(\sum_{j}\int_{\mathbb{R}^{n}}\left|\left(\partial_{x_{j}}\Lambda_{+}\right)\left(\frac{\sum\lambda_{i}E^{i}\left(\lambda_{i}u\right)}{\sum\lambda_{i}^{2}}\right)\right|^{p}dx\right)^{1/p}\nonumber\\
&+\left(\sum_{j}\int_{\mathbb{R}^{n}}\left|\Lambda_{+}\left(\circ\right)\right|^{p}dx\right)^{1/p}\nonumber\\
&+\left(\sum_{j}\int_{\mathbb{R}^{n}}\left|\left(\partial_{x_{j}}\Lambda_{-}\right)u\right|^{p}dx\right)^{1/p}+\left(\sum_{j}\int_{\mathbb{R}^{n}}\left|\Lambda_{-}\left(\partial_{x_{j}}u\right)\right|^{p}dx\right)^{1/p},\label{extension/estimate}
\end{align}
where 
\begin{align*}
\displaystyle \circ:=\frac{\left(\partial_{x_{j}}\sum\lambda_{i}E^{i}\left(\lambda_{i}u\right)\right)\left(\sum\lambda_{i}^{2}\right)-\left(\sum\lambda_{i}E^{i}\left(\lambda_{i}u\right)\right)\left(\partial_{x_{j}}\sum\lambda_{i}^{2}\right)}{\left(\sum\lambda_{i}^{2}\right)^{2}}.
\end{align*}
From Lemma \ref{fact}, the first term of (\ref{extension/estimate}) is evaluated as
\begin{align*}
&\left(\sum_{j}\int_{\mathbb{R}^{n}}\left|\left(\partial_{x_{j}}\Lambda_{+}\right)\left(\frac{\sum\lambda_{i}E^{i}\left(\lambda_{i}u\right)}{\sum\lambda_{i}^{2}}\right)\right|^{p}dx\right)^{1/p}\nonumber\\
\leq&b_{+}\left(\sum_{j}\int_{U^{*}}\left|\frac{\sum\lambda_{i}E^{i}\left(\lambda_{i}u\right)}{\sum\lambda_{i}^{2}}\right|^{p}dx\right)^{1/p}\nonumber\\
\leq&b_{+}n^{1/p}\left(\int_{U^{*}}\left|\sum\lambda_{i}E^{i}\left(\lambda_{i}u\right)\right|^{p}dx\right)^{1/p}\nonumber\\
\leq&b_{+}N^{1-1/p}n^{1/p}\left(\sum\int_{U_{i}}\left|E^{i}(\lambda_{i}u)\right|^{p}dx\right)^{1/p}\nonumber\\
\leq&b_{+}N^{1-1/p}An^{1/p}\left(\sum\int_{\Omega}\left|\lambda_{i}u\right|^{p}dx\right)^{1/p}\nonumber\\
\leq&b_{+}NAn^{1/p}\left(\int_{\Omega}\left|u\right|^{p}dx\right)^{1/p}.
\end{align*}
The second term of (\ref{extension/estimate}) is evaluated as
\begin{align}
&\left(\sum_{j}\int_{\mathbb{R}^{n}}\left|\Lambda_{+}\left(\circ\right)\right|^{p}dx\right)^{1/p}\nonumber\\
\leq&\left(\sum_{j}\int_{U^{*}}\left|\frac{\partial_{x_{j}}\sum\lambda_{i}E^{i}\left(\lambda_{i}u\right)}{\sum\lambda_{i}^{2}}\right|^{p}dx\right)^{1/p}\nonumber\\
&~~~~~~~~~~~~~~~~+\left(\sum_{j}\int_{U^{*}}\left|\frac{\left(\sum\lambda_{i}E^{i}\left(\lambda_{i}u\right)\right)\left(\partial_{x_{j}}\sum\lambda_{i}^{2}\right)}{\left(\sum\lambda_{i}^{2}\right)^{2}}\right|^{p}dx\right)^{1/p}\label{secondterm}.
\end{align}
The first term of (\ref{secondterm}) is evaluated as
\begin{align}
&\left(\sum_{j}\int_{U^{*}}\left|\frac{\partial_{x_{j}}\sum\lambda_{i}E^{i}\left(\lambda_{i}u\right)}{\sum\lambda_{i}^{2}}\right|^{p}dx\right)^{1/p}\nonumber\\
=&\left(\sum_{j}\int_{U^{*}}\left|\frac{\sum\left(\partial_{x_{j}}\lambda_{i}\right)E^{i}\left(\lambda_{i}u\right)+\sum\lambda_{i}\left(\partial_{x_{j}}E^{i}\left(\lambda_{i}u\right)\right)}{\sum\lambda_{i}^{2}}\right|^{p}dx\right)^{1/p}\nonumber\\
\leq&\left(\sum_{j}\int_{U^{*}}\left|\sum\left(\partial_{x_{j}}\lambda_{i}\right)E^{i}\left(\lambda_{i}u\right)\right|^{p}dx\right)^{1/p}\nonumber\\
&~~~~~~~~~~~~~~~~~~~~~~~~~+\left(\sum_{j}\int_{U^{*}}\left|\sum\lambda_{i}\left(\partial_{x_{j}}E^{i}\left(\lambda_{i}u\right)\right)\right|^{p}dx\right)^{1/p}.\label{second-1}
\end{align}
The first term of (\ref{second-1}) is evaluated as
\begin{align*}
&\left(\sum_{j}\int_{U^{*}}\left|\sum_{i}\left(\partial_{x_{j}}\lambda_{i}\right)E^{i}\left(\lambda_{i}u\right)\right|^{p}dx\right)^{1/p}\nonumber\\
\leq&\left(\sum_{j}N^{p-1}\sum_{i}\int_{U^{*}}\left|\left(\partial_{x_{j}}\lambda_{i}\right)E^{i}\left(\lambda_{i}u\right)\right|^{p}dx\right)^{1/p}\nonumber\\
\leq&N^{1-1/p}\left(\sum_{j}\sum_{i}\int_{U_{i}}b_{\varepsilon}^{p}\left|E^{i}\left(\lambda_{i}u\right)\right|^{p}dx\right)^{1/p}\nonumber\\
\leq&b_{\varepsilon}N^{1-1/p}\left(\sum_{j}\sum_{i}\int_{U_{i}}\left|E^{i}\left(\lambda_{i}u\right)\right|^{p}dx\right)^{1/p}\nonumber\\
\leq&b_{\varepsilon}N^{1-1/p}n^{1/p}\left(\sum_{i}\int_{U_{i}}\left|E^{i}\left(\lambda_{i}u\right)\right|^{p}dx\right)^{1/p}\nonumber\\
\leq&b_{\varepsilon}N^{1-1/p}An^{1/p}\left(\sum_{i}\int_{\Omega}\left|\lambda_{i}u\right|^{p}dx\right)^{1/p}\nonumber\\
\leq&b_{\varepsilon}NAn^{1/p}\left(\int_{\Omega}\left|u\right|^{p}dx\right)^{1/p}.
\end{align*}
The second term of (\ref{second-1}) is evaluated as
\begin{align*}
&\left(\sum_{j}\int_{U^{*}}\left|\sum_{i}\lambda_{i}\left(\partial_{x_{j}}E^{i}\left(\lambda_{i}u\right)\right)\right|^{p}dx\right)^{1/p}\nonumber\\
\leq&\left(\sum_{j}N^{p-1}\sum_{i}\int_{U^{*}}\left|\lambda_{i}\left(\partial_{x_{j}}E^{i}\left(\lambda_{i}u\right)\right)\right|^{p}dx\right)^{1/p}\nonumber\\
\leq&N^{1-1/p}\left(\sum_{i}\sum_{j}\int_{U_{i}}\left|\partial_{x_{j}}E^{i}\left(\lambda_{i}u\right)\right|^{p}dx\right)^{1/p}\nonumber\\
\leq&N^{1-1/p}A'\left(\sum_{i}\sum_{j}\int_{U_{i}}\left|\partial_{x_{j}}\left(\lambda_{i}u\right)\right|^{p}dx\right)^{1/p}\nonumber\\
=&N^{1-1/p}A'\left(\sum_{i}\sum_{j}\int_{\Omega}\left|\left(\partial_{x_{j}}\lambda_{i}\right)u+\lambda_{i}\left(\partial_{x_{j}}u\right)\right|^{p}dx\right)^{1/p}\nonumber\\
\leq&N^{1-1/p}A'\left(\sum_{i}\int_{\Omega}\sum_{j}\left|\left(\partial_{x_{j}}\lambda_{i}\right)u\right|^{p}dx\right)^{1/p}\nonumber\\
&~~~~~~~~~~~~~~+N^{1-1/p}A'\left(\sum_{i}\int_{\Omega}\sum_{j}\left|\lambda_{i}\left(\partial_{x_{j}}u\right)\right|^{p}dx\right)^{1/p}\nonumber\\
\leq&NA'\left\{\left(\sum_{j}b_{\varepsilon}^{p}\int_{\Omega}\left|u\right|^{p}dx\right)^{1/p}+\left(\int_{\Omega}\sum_{j}\left|\partial_{x_{j}}u\right|^{p}dx\right)^{1/p}\right\}\nonumber\\
\leq&NA'\left\{b_{\varepsilon}n^{1/p}\left(\int_{\Omega}\left|u\right|^{p}dx\right)^{1/p}+\left(\int_{\Omega}\sum_{j}\left|\partial_{x_{j}}u\right|^{p}dx\right)^{1/p}\right\}.
\end{align*}
The second term of (\ref{secondterm}) is evaluated as
\begin{align*}
&\left(\sum_{j}\int_{U^{*}}\left|\frac{\left(\sum\lambda_{i}E^{i}\left(\lambda_{i}u\right)\right)\left(\partial_{x_{j}}\sum\lambda_{i}^{2}\right)}{\left(\sum\lambda_{i}^{2}\right)^{2}}\right|^{p}dx\right)^{1/p}\nonumber\\
\leq&\left(\sum_{j}\int_{U^{*}}\left|\frac{\left(\sum\lambda_{i}E^{i}\left(\lambda_{i}u\right)\right)\left(2\sum\frac{\partial_{x_{j}}\lambda_{i}}{\lambda_{i}}\lambda_{i}^{2}\right)}{\left(\sum\lambda_{i}^{2}\right)^{2}}\right|^{p}dx\right)^{1/p}\nonumber\\
\leq&\left(\sum_{j}\int_{U^{*}}\left|\frac{\left(\sum\lambda_{i}E^{i}\left(\lambda_{i}u\right)\right)\left(2b_{\varepsilon}\sum\lambda_{i}^{2}\right)}{\left(\sum\lambda_{i}^{2}\right)^{2}}\right|^{p}dx\right)^{1/p}\nonumber\\
\leq&2b_{\varepsilon}\left(\sum_{j}\int_{U^{*}}\left|\sum\lambda_{i}E^{i}\left(\lambda_{i}u\right)\right|^{p}dx\right)^{1/p}\nonumber\\
\leq&2b_{\varepsilon}N^{1-1/p}n^{1/p}\left(\sum\int_{U_{i}}\left|E^{i}(\lambda_{i}u)\right|^{p}dx\right)^{1/p}\nonumber\\
\leq&2b_{\varepsilon}N^{1-1/p}An^{1/p}\left(\sum\int_{\Omega}\left|\lambda_{i}u\right|^{p}dx\right)^{1/p}\nonumber\\
\leq&2b_{\varepsilon}NAn^{1/p}\left(\int_{\Omega}\left|u\right|^{p}dx\right)^{1/p}.
\end{align*}
From the above evaluations, we have
\begin{align*}
&\left\|\nabla\left(Eu\right)\right\|_{L^{p}\left(\mathbb{R}^{n}\right)}\nonumber\\
\leq&b_{+}NAn^{1/p}\left(\int_{\Omega}\left|u\right|^{p}dx\right)^{1/p}+Nb_{\varepsilon}An^{1/p}\left(\int_{\Omega}\left|u\right|^{p}dx\right)^{1/p}\nonumber\\
&+NA'\left\{b_{\varepsilon}n^{1/p}\left(\int_{\Omega}\left|u\right|^{p}dx\right)^{1/p}+\left(\int_{\Omega}\sum_{j}\left|\partial_{x_{j}}u\right|^{p}dx\right)^{1/p}\right\}\nonumber\\
&+2Nb_{\varepsilon}An^{1/p}\left(\int_{\Omega}\left|u\right|^{p}dx\right)^{1/p}+b_{-}n^{1/p}\left(\int_{\Omega}\left|u\right|^{p}dx\right)^{1/p}\nonumber\\
&+\left(\int_{\Omega}\sum_{j}\left|\partial_{x_{j}}u\right|^{p}dx\right)^{1/p}\nonumber\\
=&\left(NA'+1\right)\left(\int_{\Omega}\sum_{j}\left|\partial_{x_{j}}u\right|^{p}dx\right)^{1/p}\nonumber\\
&~~~~~~~~~~~+\left(b_{+}NA+b_{\varepsilon}NA+2b_{\varepsilon}NA+b_{\varepsilon}NA'+b_{-}\right)n^{1/p}\left(\int_{\Omega}\left|u\right|^{p}dx\right)^{1/p}\nonumber\\
=&\left(NA'+1\right)\left\|\nabla u\right\|_{L^{p}\left(\Omega\right)}+b_{\varepsilon}\left(6NA+NA'+3\right)n^{1/p}\left\|u\right\|_{L^{p}\left(\Omega\right)}.
\end{align*}
Hence, the inequality (\ref{ap/condision}) holds for
\begin{align*}
A_{p}\left(\Omega\right)=\left\{\begin{array}{l}
\left(NA'+1\right),~~R\leq\gamma,\\
b_{\varepsilon}\left(6NA+NA'+3\right)n^{1/p}/\gamma,~~R>\gamma,
\end{array}\right.
\end{align*}
where $R:=b_{\varepsilon}\left(6NA+NA'+3\right)n^{1/p}/\left(NA'+1\right)$.
\end{proof}

\begin{rem}
The value $A_{p}\left(\Omega\right)$ derived by Theorem \ref{main} monotonically decreases with decreasing $\xi \in (0,1)$.
Moreover, $A_{p}\left(\Omega\right)\rightarrow A_{p}\left(\Omega\right)|_{\xi=0}~(\xi\downarrow 0)$ holds.
Therefore, $ A_{p}\left(\Omega\right)|_{\xi=0}+\delta$ with any positive number $\delta$ becomes an upper bound of the operator norm, while the range of $\xi$ is $(0,1)$.
\end{rem}
The operator norm derived by Theorem \ref{main} leads bounds for the embedding constant as in the following corollary.
\begin{coro}\label{embedding/coro}
For given $n\in\{2,3\cdots\}$ and $p\in(n/(n-1),\infty)$,
let $T_{p}$ be a constant in the classical Sobolev inequality, i.e.,
$\left\|u\right\|_{L^{p}\left(\mathbb{R}^{n}\right)}\leq T_{p}\left\|\nabla u\right\|_{L^{q}\left(\mathbb{R}^{n}\right)}$
for all
$u\in W^{1,q}(\mathbb{R}^{n})$,
where $q=np/(n+p)$.
Moreover, let $\Omega\subset \mathbb{R}^{n}$ be a domain with minimally smooth boundary.
Then,
\begin{align}
\left\|u\right\|_{L^{p}(\Omega)}\leq C_{p}(\Omega)\left\|u\right\|_{W^{1,q}(\Omega)},\ \forall u\in W^{1,q}(\Omega)\label{cp/thoe2}
\end{align}
holds for
\begin{align*}
C_{p}\left(\Omega\right)=2^{\frac{q-1}{q}}T_{p}A_{q}\left(\Omega\right).
\end{align*}
Here, $\left\|\cdot\right\|_{W^{1,q}(\Omega)}$ denotes the $\sigma$-weighted $W^{1,q}$ norm $(\ref{sigmanorm})$ for given $\sigma>0$, and $A_{q}\left(\Omega\right)$ is the upper bound for the operator norm derived by Theorem \ref{main} with $\gamma=\sigma^{1/q}$.
\end{coro}
\begin{proof}
We have
\begin{align}
\left\|u\right\|_{L^{p}\left(\Omega\right)}&\leq\left\|Eu\right\|_{L^{p}\left(\mathbb{R}^{n}\right)}\nonumber\\
&\leq T_{p}\left\|\nabla Eu\right\|_{L^{q}\left(\mathbb{R}^{n}\right)}\nonumber\\
&\leq T_{p}A_{q}\left(\Omega\right)\left(\left\|\nabla u\right\|_{L^{q}\left(\Omega\right)}+\sigma^{1/q}\left\|u\right\|_{L^{q}\left(\Omega\right)}\right)\nonumber\\
&\leq 2^{\frac{q-1}{q}}T_{p}A_{q}\left(\Omega\right)\left\|u\right\|_{W^{1,q}(\Omega)}\label{proof/coro}
\end{align}
for all $u\in W^{1,q}(\Omega)$.
\end{proof}
\begin{rem}
The constant $C_{p}'\left(\Omega\right)$ such that $\left\|u\right\|_{L^{p}(\Omega)}\leq C_{p}'\left(\Omega\right)\left\|u\right\|_{H^{1}(\Omega)}$ for all $u\in H^{1}(\Omega)$ is also important especially for verified numerical computation method and compute-assisted proof for PDEs summarized in, e.g., {\rm \cite{nakao2001numerical, plum2001computer, plum2008, takayasu2013verified}}.
We can obtain a formula giving a concrete value of $C_{p}'\left(\Omega\right)$ with additional assumptions for $\Omega$ and $p$ (see Corollary \ref{coroap}).
\end{rem}

\section{Examples}\label{result}
In this section, we present some examples of estimating the embedding constant $C_{p}\left(\Omega\right)$ defined in (\ref{purpose}) using Theorem \ref{main} and Corollary \ref{embedding/coro}.
Through out this section, we set $\rho$ as the mollifier defined in (\ref{sp/mol}) and set $\sigma=1$.
\subsection{Calculation of the constants}
\begin{figure}[t]
\begin{minipage}{0.5\hsize}
\begin{center}
  \includegraphics[height=60mm]{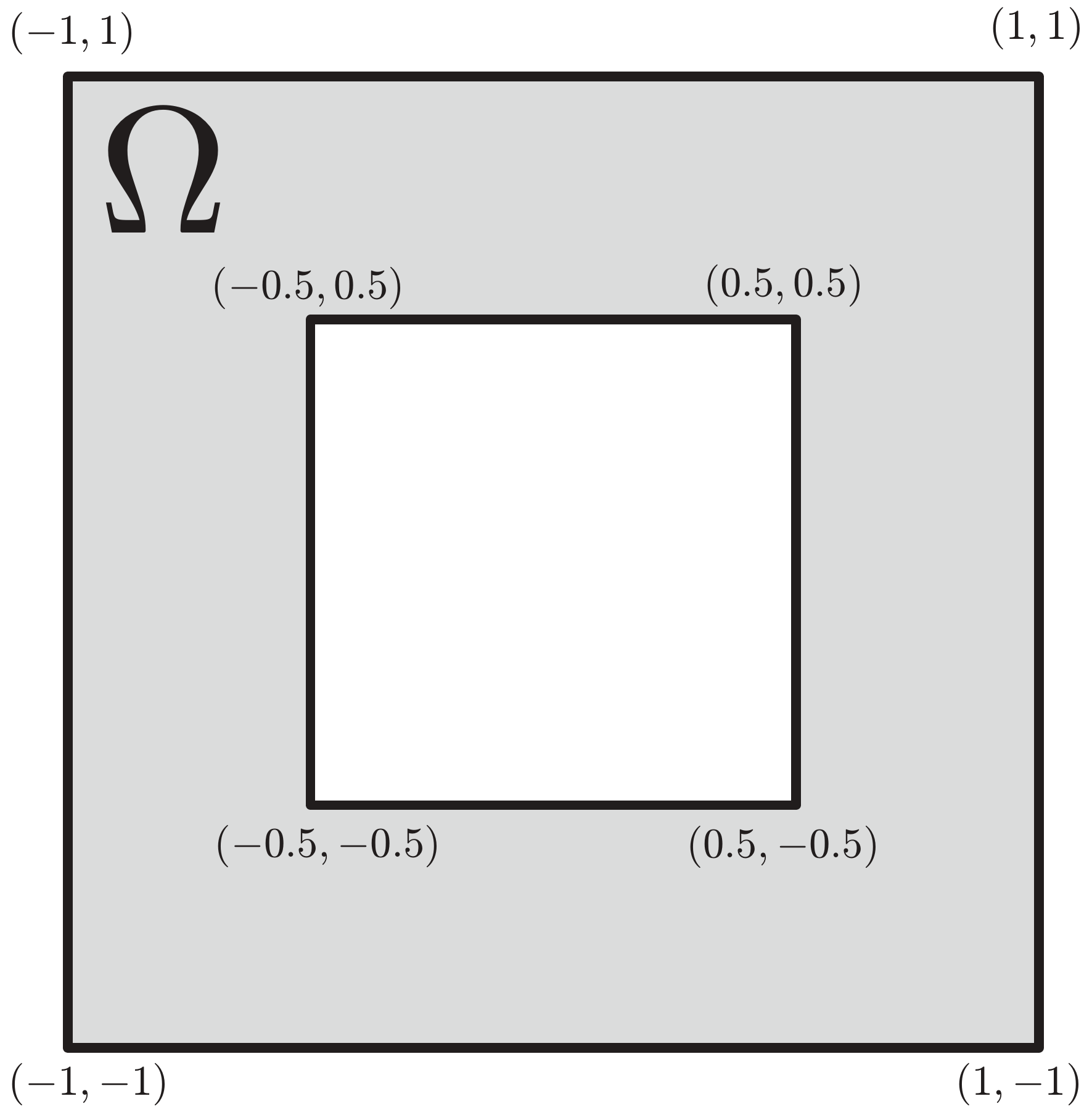}\\
{\Large (a)}
\end{center}
\end{minipage}
\begin{minipage}{0.5\hsize}
\begin{center}
  \includegraphics[height=60mm]{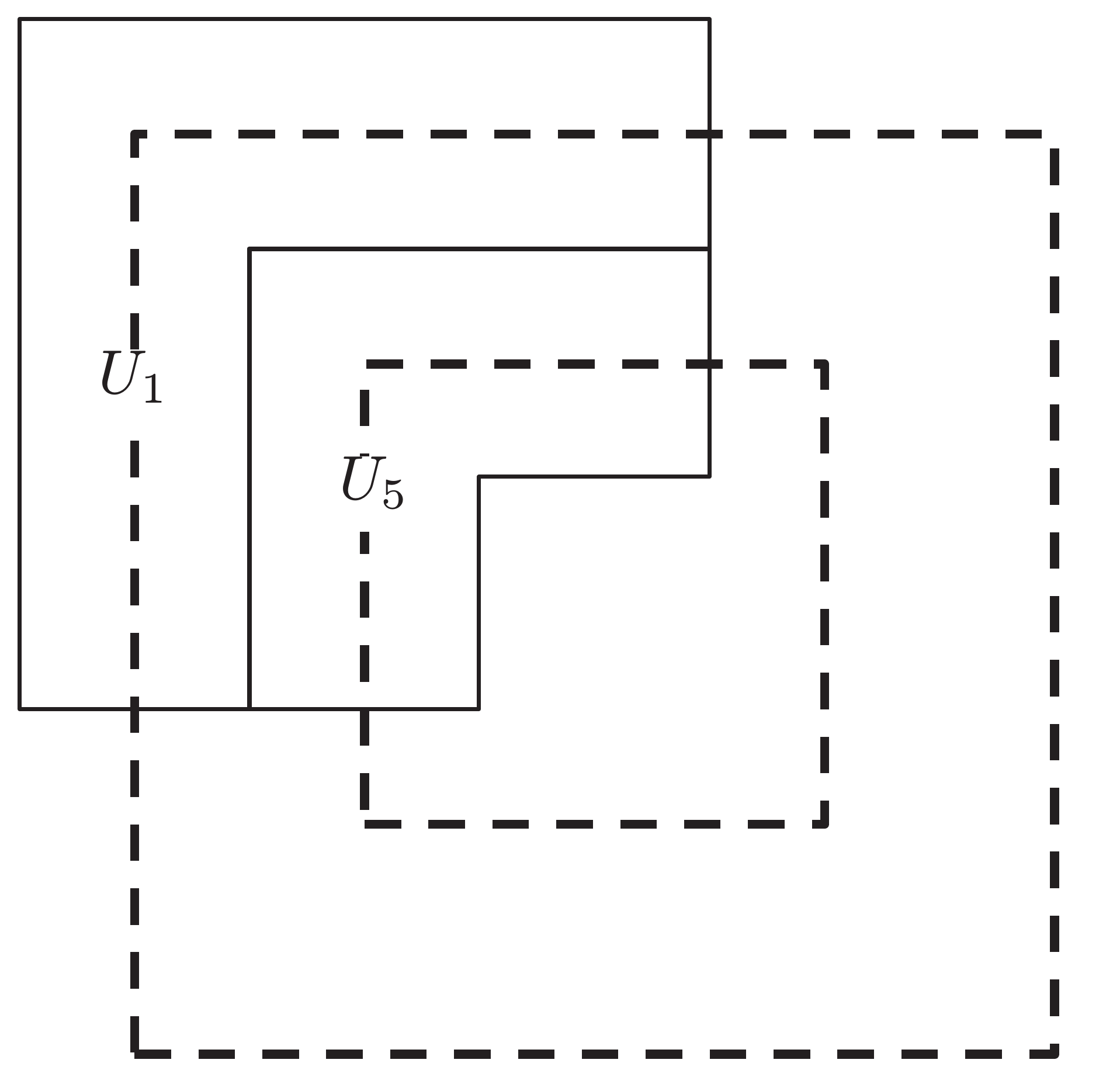}\\
~~{\Large (b)}
\end{center}
\end{minipage}\\
\begin{minipage}{\hsize}
\begin{center}
  \includegraphics[height=60mm]{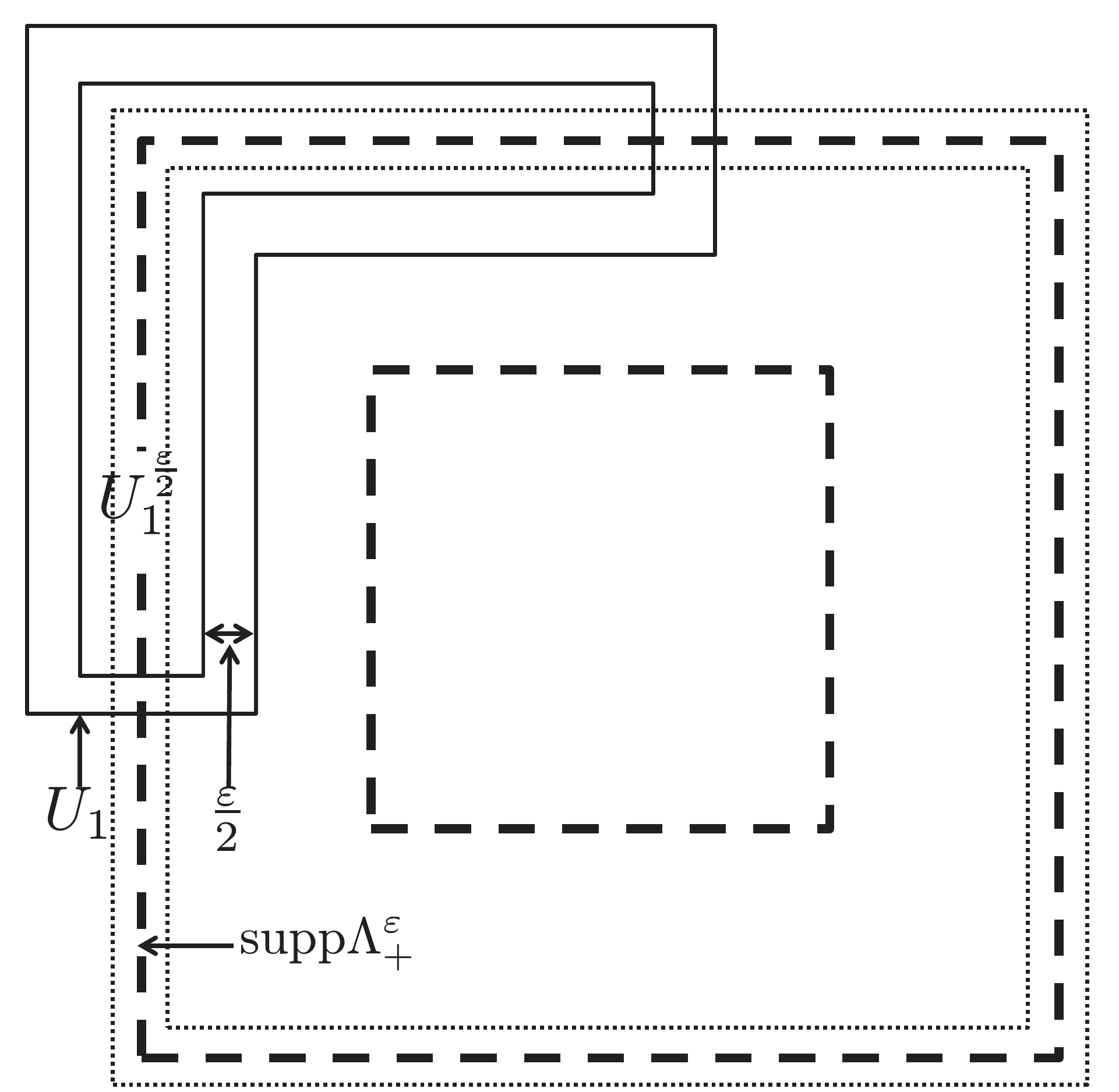}\\[3pt]
~~~{\Large (c)}
\end{center}
\end{minipage}\\
\caption{(a): The domain $\Omega$ of Example A.
(b) and (c): open sets $U_{i}$~$(i=1,2,\cdots,8)$.}
\label{hole/domain}
\end{figure}
The constants $A_{0},\ A_{1},\ P$, and $b_{\varepsilon}$ in Lemma \ref{sp/lemma} and Theorem \ref{main} were numerically calculated.
All computations were carried out on a computer with Intel Core i7 860 CPU 2.80 GHz, 16.0 GB RAM, Windows 7, and MATLAB 2012b.
Since all rounding errors were strictly estimated using INTLAB version 6 \cite{rump1999book}, a toolbox for verified numerical computations,
the accuracy of all results is mathematically guaranteed.

The constants $A_{0}$ and $A_{1}$ can be respectively computed as
\begin{align}
A_{0}&=\displaystyle \sup\left\{\left|t^{2}\psi\left(t\right)\right|\ :\ t\geq 1\right\}~~{\rm and}~~A_{1}=\sup\left\{\left|t^{3}\psi\left(t\right)\right|\ :\ t\geq 1\right\}
\end{align}
with the function $\psi:\mathbb{R}\rightarrow \mathbb{R}$ constructing the extension operator (\ref{extension/sp}) which satisfies the property (\ref{phi/prop}).
For example, the function
\begin{eqnarray}
\displaystyle \psi\left(t\right):=\frac{e}{\pi t}{\rm Im}\left(e^{-\omega\left(t-1\right)^{\frac{1}{4}}}\right),\ \ \ \omega=C_{\omega}e^{-\frac{i\pi}{4}}=\frac{C_{\omega}}{\sqrt{2}}\left(1-i\right)\label{concrete/phi}
\end{eqnarray}
satisfies that property for any $C_{\omega}>0$;
a simple proof can be seen in, e.g., \cite{edmunds1987book, stein1970book}.
For the function $\psi$ in (\ref{concrete/phi}) with $C_{\omega}=4.83$, we derived the following estimation results:
\begin{align*}
A_{0}\in[12.8860,\ 12.8861]~~{\rm and}~~A_{1}\in[12.9325,\ 12.9326].
\end{align*}
Moreover, recall that $b_{\varepsilon}$ is a positive number satisfying
\begin{align}
b_{\varepsilon}&\geq\int_{\mathbb{R}^{n}}\left|\partial_{x_{j}}\rho_{\frac{1}{4}\varepsilon}\left(x\right)\right|dx~\left(=\displaystyle \frac{4}{\varepsilon}\int_{\mathbb{R}^{n}}\left|\partial_{x_{1}}\rho\left(x\right)\right|dx\right).\label{bep}
\end{align}
For the mollifier defined in (\ref{sp/mol}), the bounds for the integration in (\ref{bep}) is independent of the index $j$.
Furthermore, one of the concrete values of $P$ can be derived by (\ref{fra/Palpha}) with the condition $|\alpha|=1$, i.e., it can be computed as
\begin{align*}
P&=\displaystyle \int_{\mathbb{R}^{n}}\left\{\left(n-1\right)\rho_{*}\left(\left|x\right|\right)+\left|x\right|\rho_{*}'\left(\left|x\right|\right)\right\}\left(1-\left|x\right|\right)^{-1}dx.
\end{align*}
Using verified numerical computation, we derived the following estimation results:
\begin{align*}
\displaystyle \int_{\mathbb{R}^{2}}\left|\partial_{x_{1}}\rho\left(x\right)\right|dx&\in[1.86412, 1.92770]~~{\rm and}~~P\in[7.45592,\ 7.50131]
\end{align*}
for the case of $n=2$.

\subsection{Examples of estimating the embedding constant}\label{example}
Here, we present estimation results for the following two concrete domains:
\subsubsection*{Example~A}
Let $\Omega\subset \mathbb{R}^{2}$ be the domain as in Fig.~\ref{hole/domain} (a).
We set $\left\{U_{i}\right\}_{i\in \mathbb{N}}$ as follows:
we first define the two sets among $U_{i}$'s as in Fig.~\ref{hole/domain} (b);
then, $U_{i}$'s $(i=1,2,\cdots,8)$ were obtained by symmetry reflections;
finally, we defined the other $U_{i}$'s $(i=9,10,11,\cdots)$ as empty sets.
In this case, we chose $M=1,\ N=2$, and $\varepsilon=0.25$.
One can find in Fig.~\ref{hole/domain} (c) that these constants satisfy the required conditions mentioned in Theorem \ref{main}.

Figure~\ref{hole/graph} (a) shows the relationship between $\tau$ and $A_{q}\left(\Omega\right)$ in the cases of $p=4,\ 6,$ and $8$; recall that $q=2p/(2+p)$.
One can observe that $A_{q}\left(\Omega\right)$ first decreases with increasing $\tau$, then it reaches a minimum point, and thereafter it monotonically increases with increasing $\tau$.
The relationship between $p$ and the value of $\tau$ minimizing $A_{q}\left(\Omega\right)$ can be seen in Fig.~\ref{hole/graph} (b).
For example, in the cases of $p=4,\ 6,$ and $8$, each $A_{q}\left(\Omega\right)$ is minimized at the points $\tau\approx 8.12,$~$5.83,$~and $5.06,$ respectively.

Figure~\ref{hole/graph} (c) shows the relationship between $p$ and $C_{p}\left(\Omega\right)$;
we chose $\tau$ which makes $A_{q}\left(\Omega\right)$ (and $C_{p}\left(\Omega\right)$) as small as possible.
Recall that all results in Fig.~\ref{hole/graph} were mathematically guaranteed with verified numerical computation.
\begin{figure}[t]
\begin{minipage}{0.5\hsize}
\begin{center}
  \includegraphics[height=50mm]{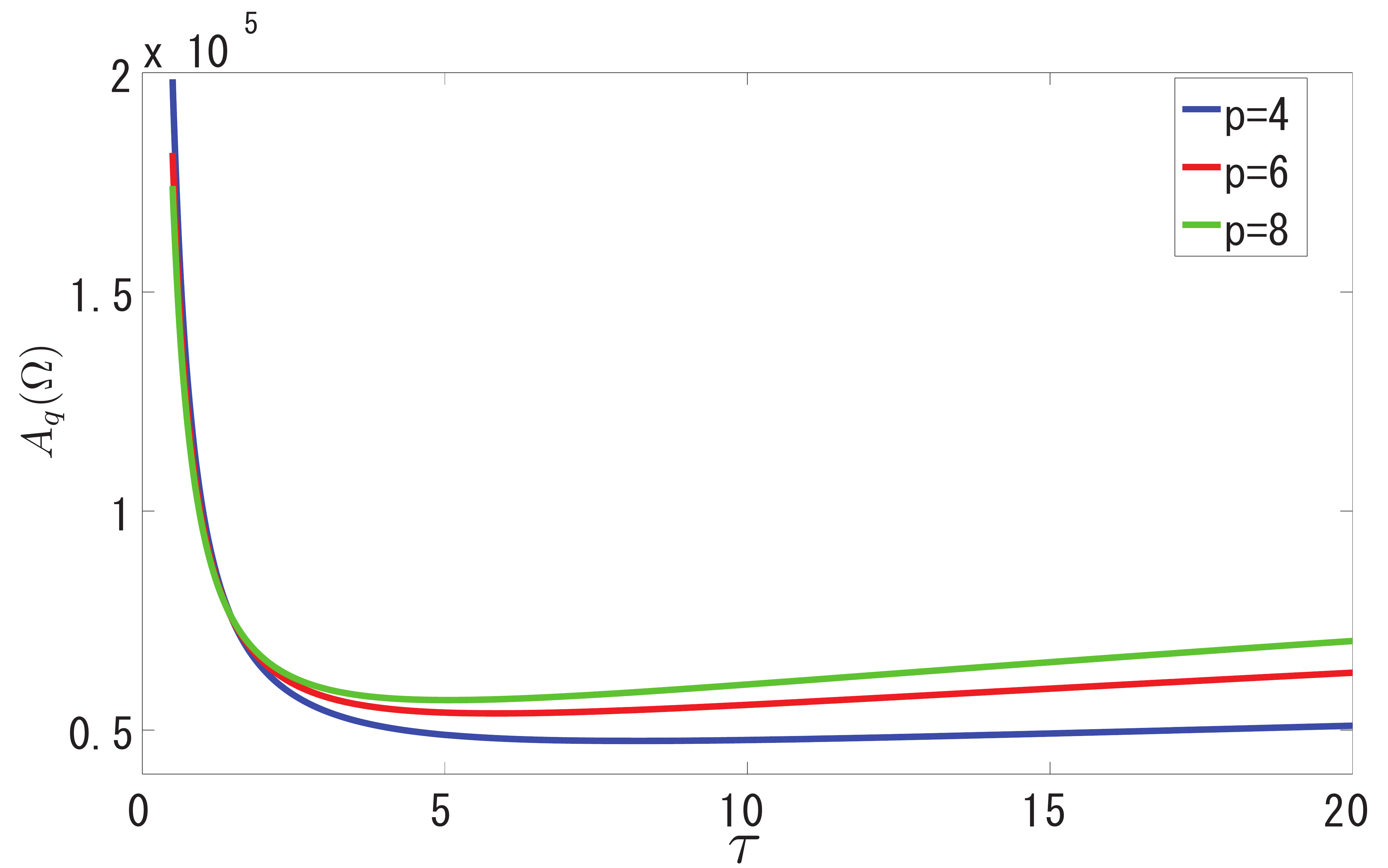}\\
~~~\,{\Large (a)}
\end{center}
\end{minipage}
\begin{minipage}{0.5\hsize}
\begin{center}
  \includegraphics[height=50mm]{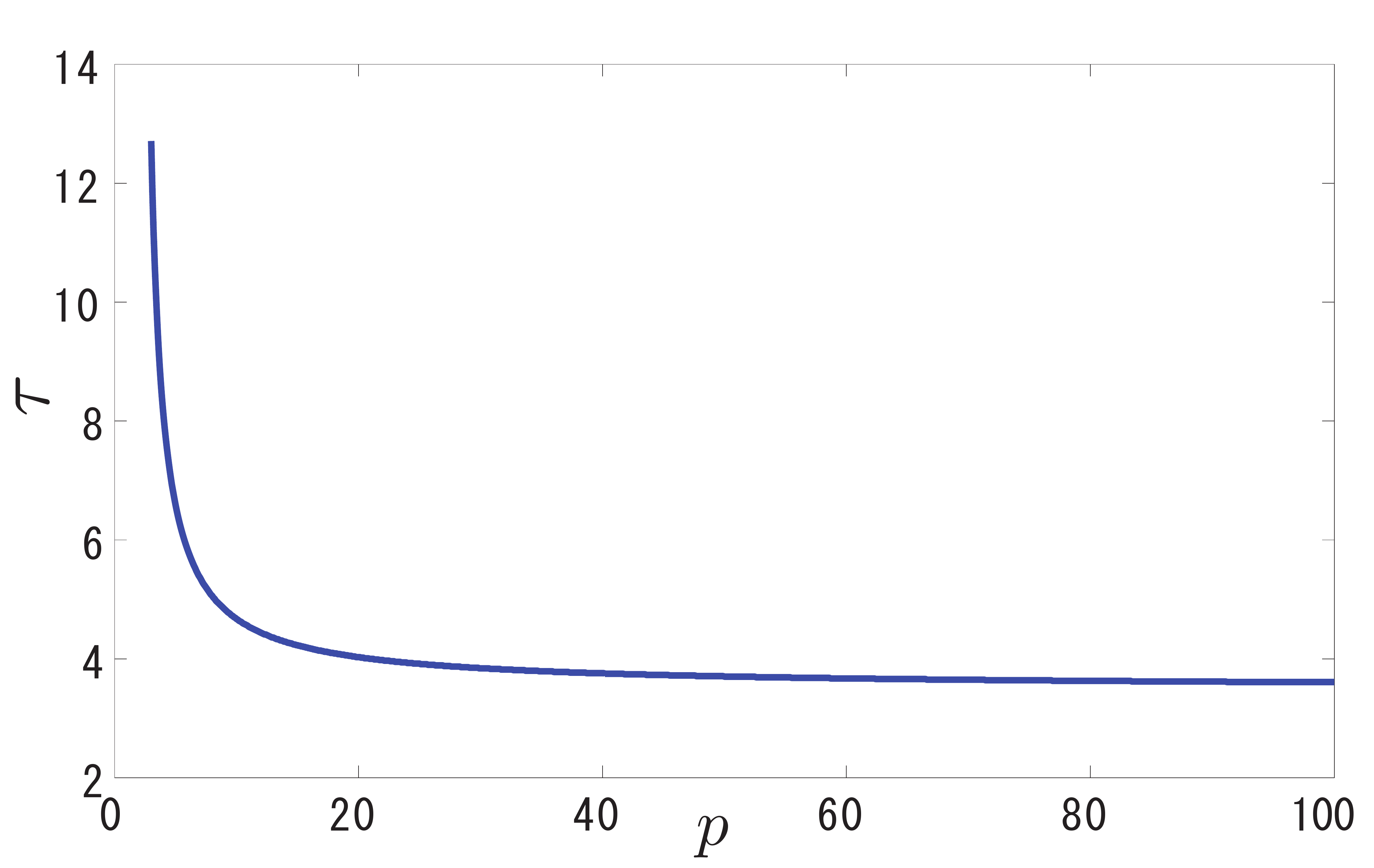}\\
~~\,{\Large (b)}
\end{center}
\end{minipage}
\begin{minipage}{\hsize}
\begin{center}
  \includegraphics[height=50mm]{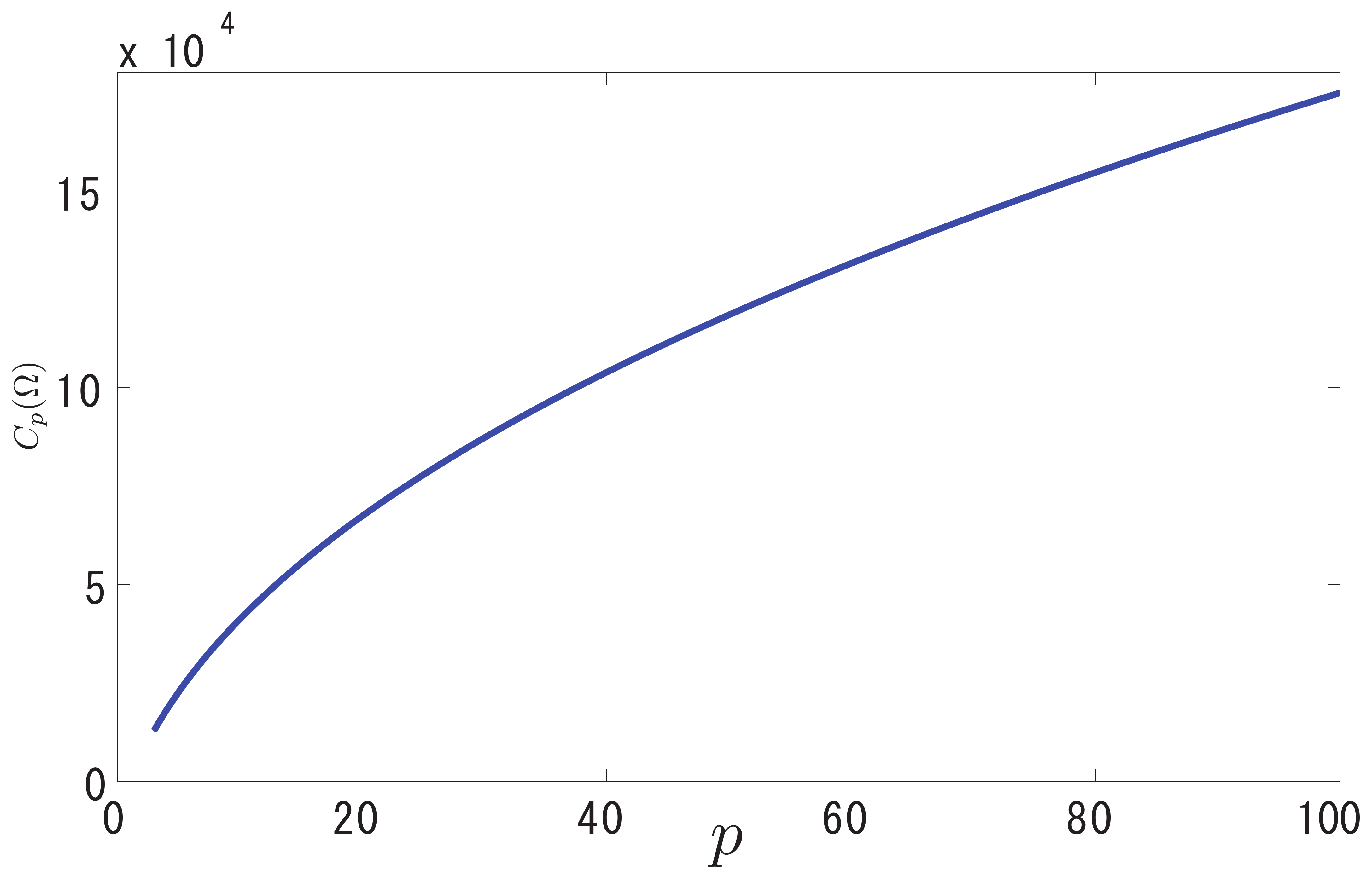}\\
~~\,{\Large (c)}
\end{center}
\end{minipage}\\
\caption{(a): The relationship between $\tau$ and $A_{q}\left(\Omega\right)$ with $p=4,$~$6,$~and~$8$.
(b): between $p$ and $\tau$ minimizing $A_{q}\left(\Omega\right)$.
(c): between $p$ and $C_{p}\left(\Omega\right)$.}
\label{hole/graph}
\end{figure}
\subsubsection*{Example~B}
Let $\Omega\subset \mathbb{R}^{2}$ be the domain as in Fig.~\ref{ufo} (a), of which boundary is composed of five semicircles and a straight line.
We set $\left\{U_{i}\right\}_{i\in \mathbb{N}}$ as follows:
we first set $U_{i}$'s $(i=1,2,\cdots,6)$ as in Fig.~\ref{ufo} (b)--(d);
then, we got the other $U_{i}$'s $(i=7,8,\cdots,10)$ by symmetrical reflection;
the other $U_{i}$'s $(i=11,12,\cdots)$ were defined as empty sets.
In this case, we chose $M=1,\ N=2$, and $\varepsilon=2\sin(\pi/8)/\{\sin(\pi/8)+1\}$.
The selection of $\varepsilon$ depends on the smallest semicircle that composes the boundary of $\Omega$.
One can find in Fig.~\ref{howep} that $\varepsilon=2\sin(\pi/8)/\{\sin(\pi/8)+1\}$ satisfies the required condition in Theorem \ref{main}.
The graphs of $A_{q}\left(\Omega\right),\ \tau$ minimizing $A_{q}\left(\Omega\right)$, and $C_{p}\left(\Omega\right)$ are also displayed in Fig.~\ref{ufo/graph}.

\begin{figure}[h]
\begin{minipage}{0.5\hsize}
\begin{center}
  \includegraphics[height=60mm]{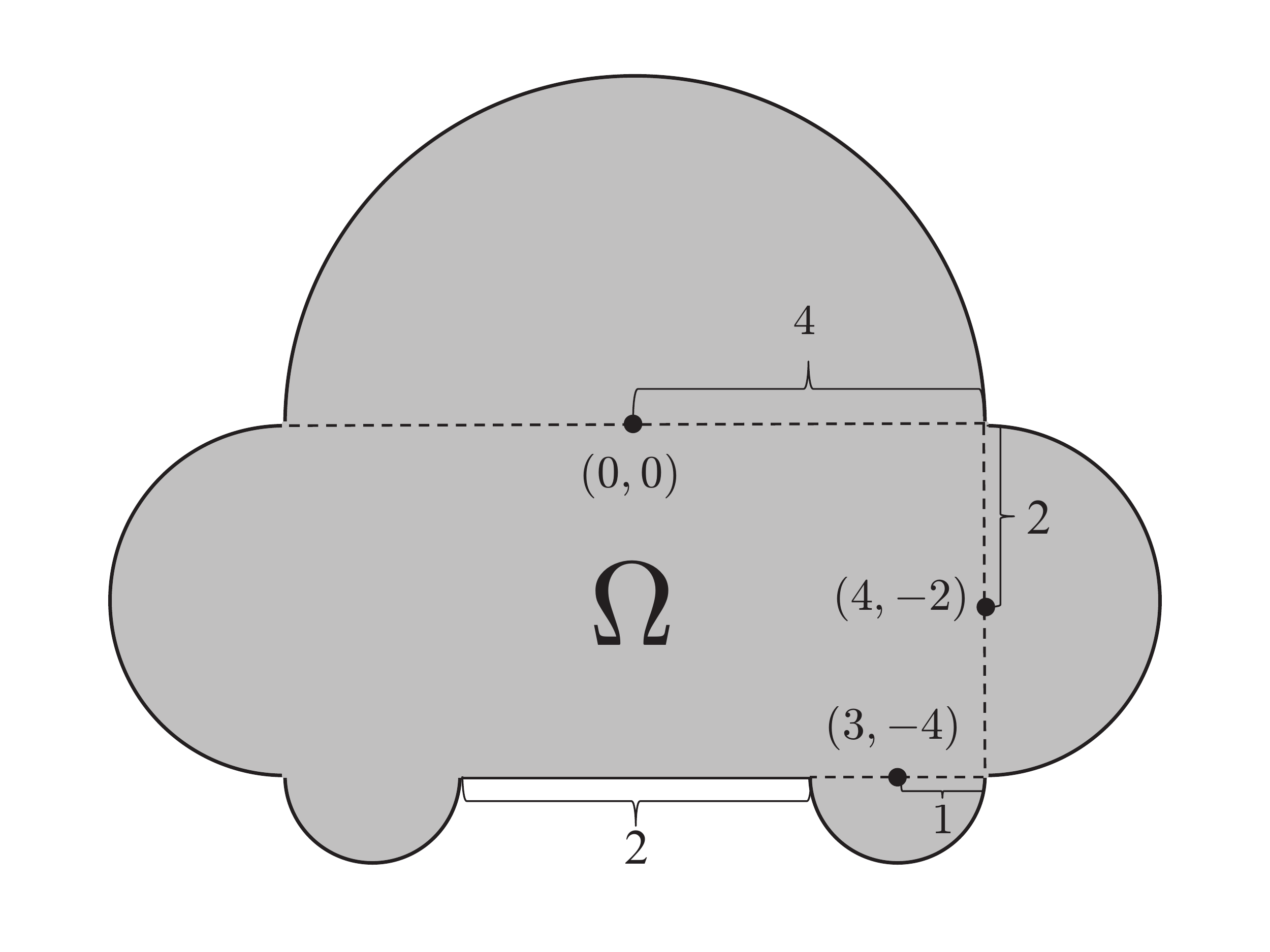}\\[-8pt]
{\Large (a)}
\end{center}
\end{minipage}
\begin{minipage}{0.5\hsize}
\begin{center}
\includegraphics[height=60mm]{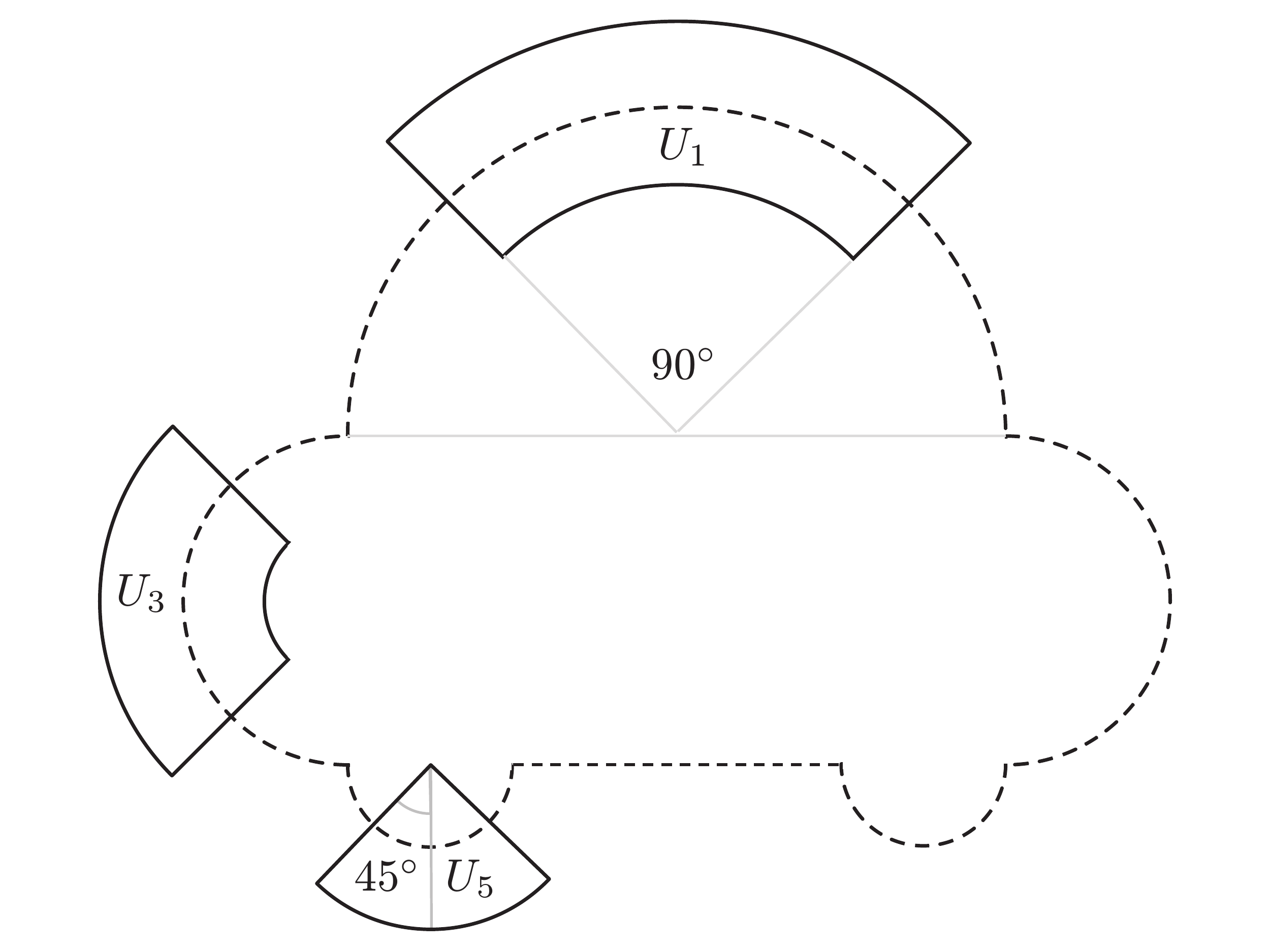}\\[-8pt]
~~~~{\Large (b)}
\end{center}
\end{minipage}\\[9pt]
\begin{minipage}{0.5\hsize}
\begin{center}
  \includegraphics[height=60mm]{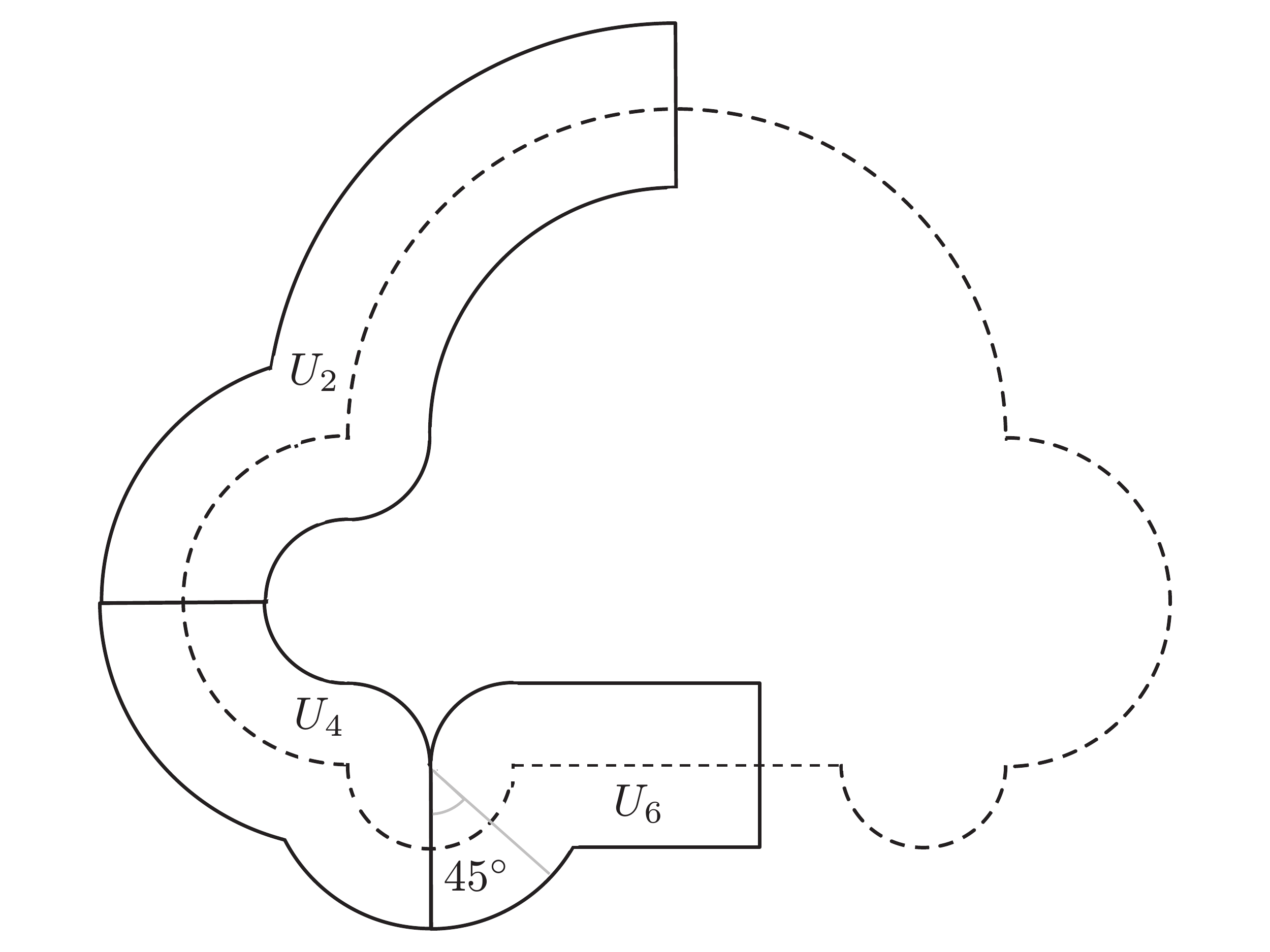}\\
\,{\Large (c)}
\end{center}
\end{minipage}
\hspace{3mm}\begin{minipage}{0.5\hsize}
\begin{center}
  \includegraphics[height=60mm]{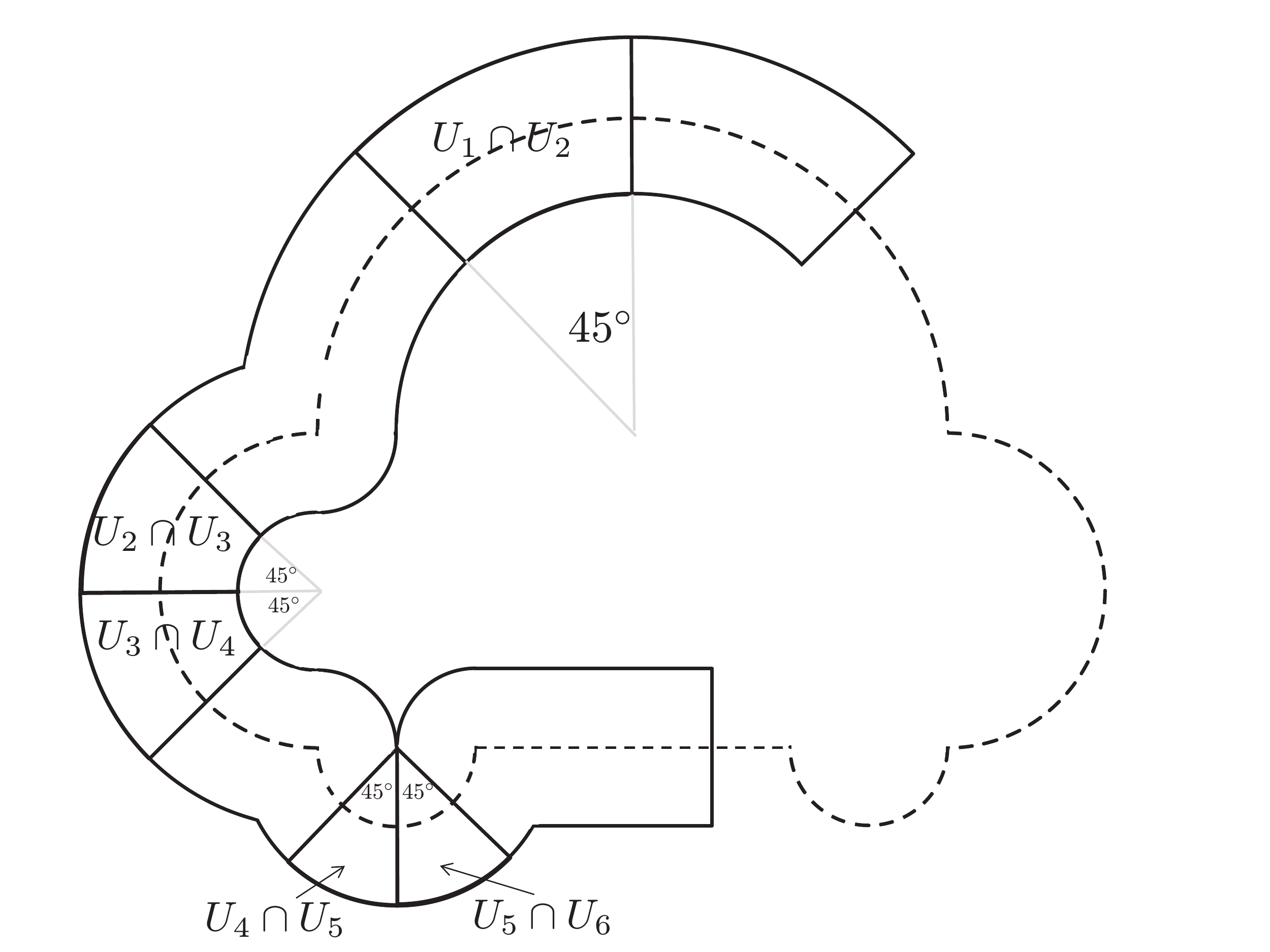}\\
{\Large (d)}
\end{center}
\end{minipage}
\vspace{6pt}\caption{(a): the domain $\Omega$ of Example B.
(b)--(d): $U_{i}$ $(i=1,2,\cdots,6)$ are displayed; the other $U_{i}\ (i=7,8,\cdots,10)$ can be obtained by symmetrical reflection.}
\label{ufo}
\end{figure}

\begin{figure}[p]
\begin{center}
  \includegraphics[width=70mm]{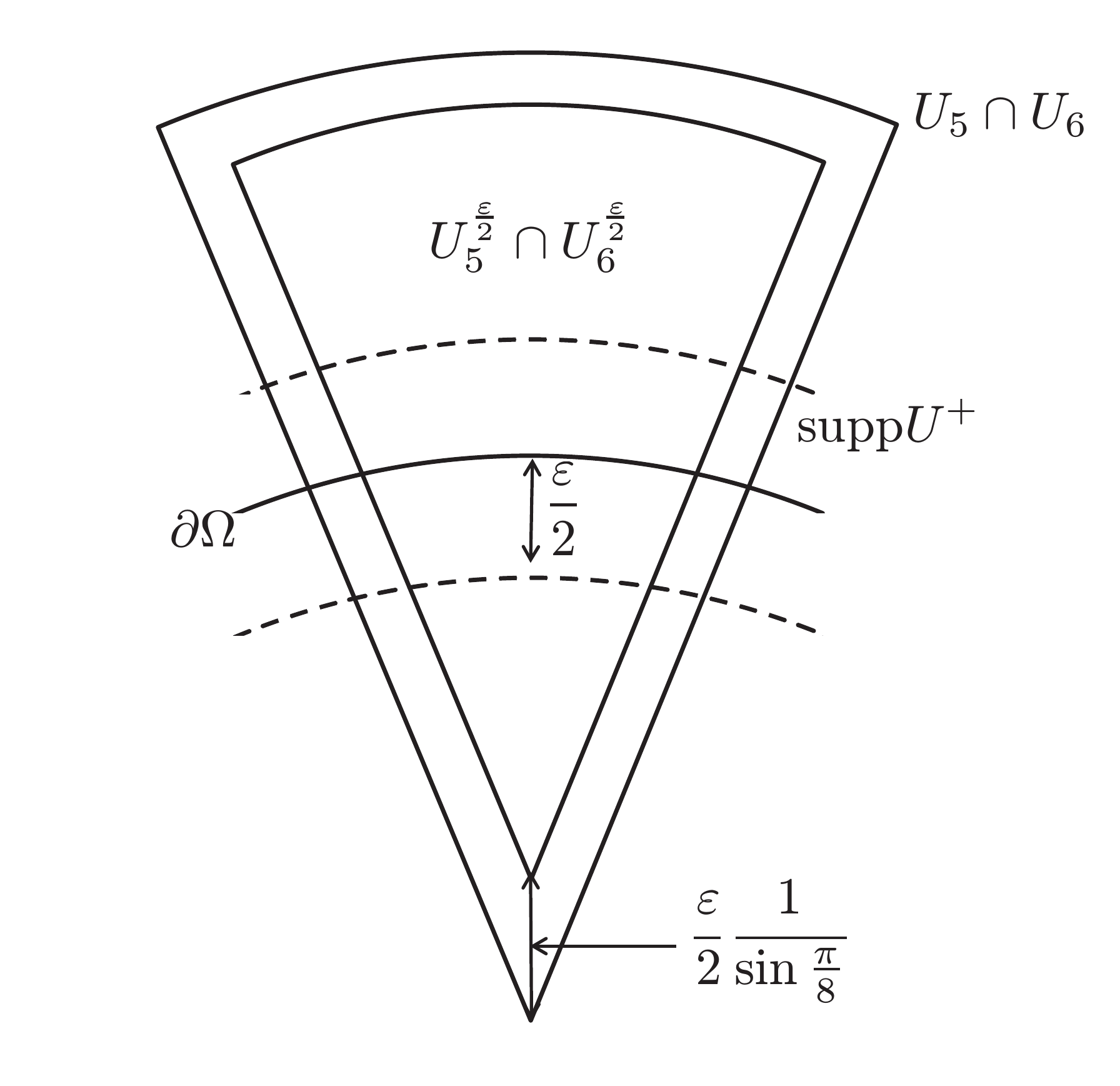}
\end{center}
\caption{How to determine $\varepsilon$.}
\label{howep}
\end{figure}

\begin{figure}[p]
\begin{minipage}{0.5\hsize}
\begin{center}
  \includegraphics[height=50mm]{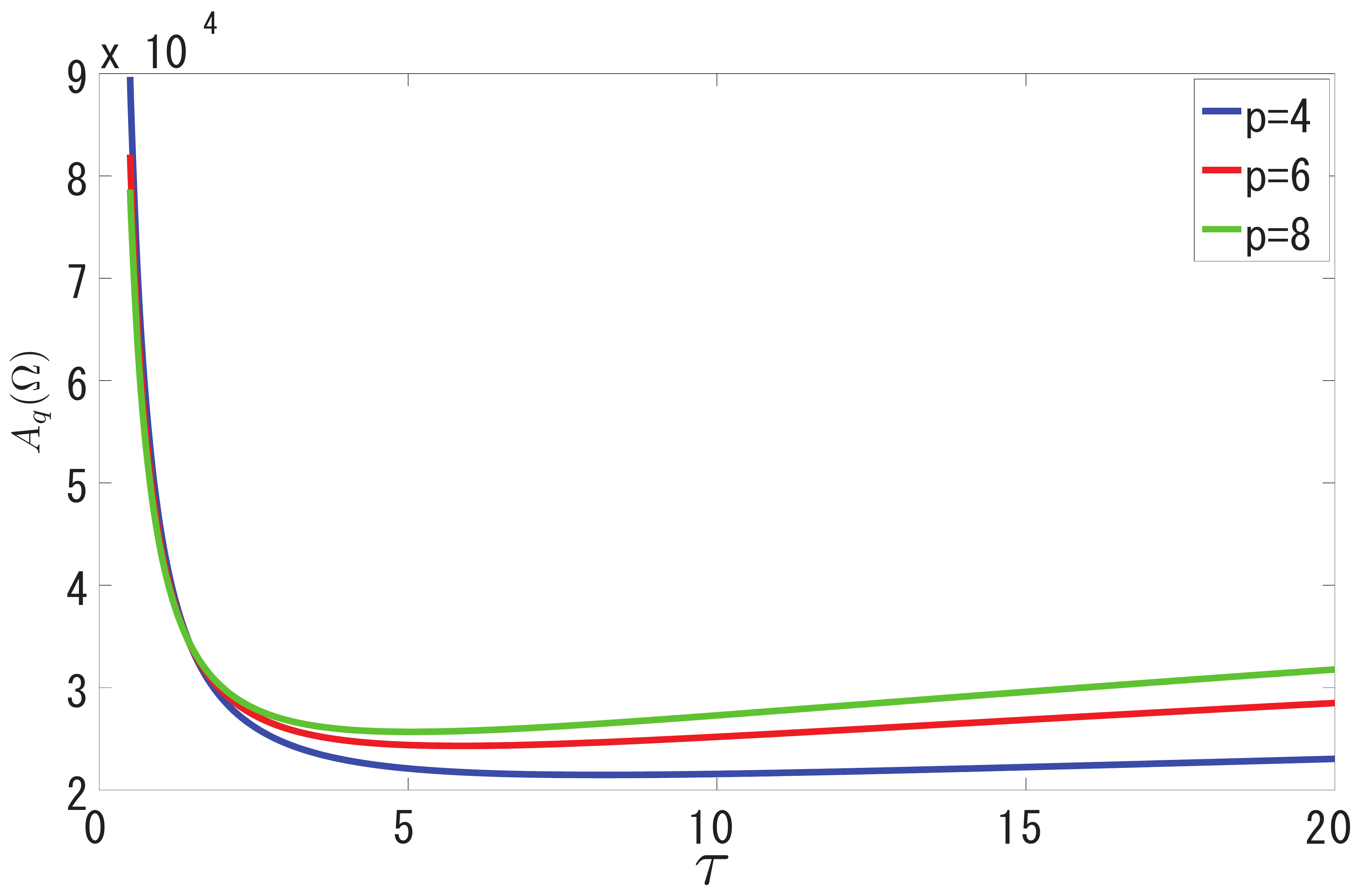}\\[2pt]
~~{\Large (a)}
\end{center}
\end{minipage}
\begin{minipage}{0.5\hsize}
\begin{center}
  \includegraphics[height=50mm]{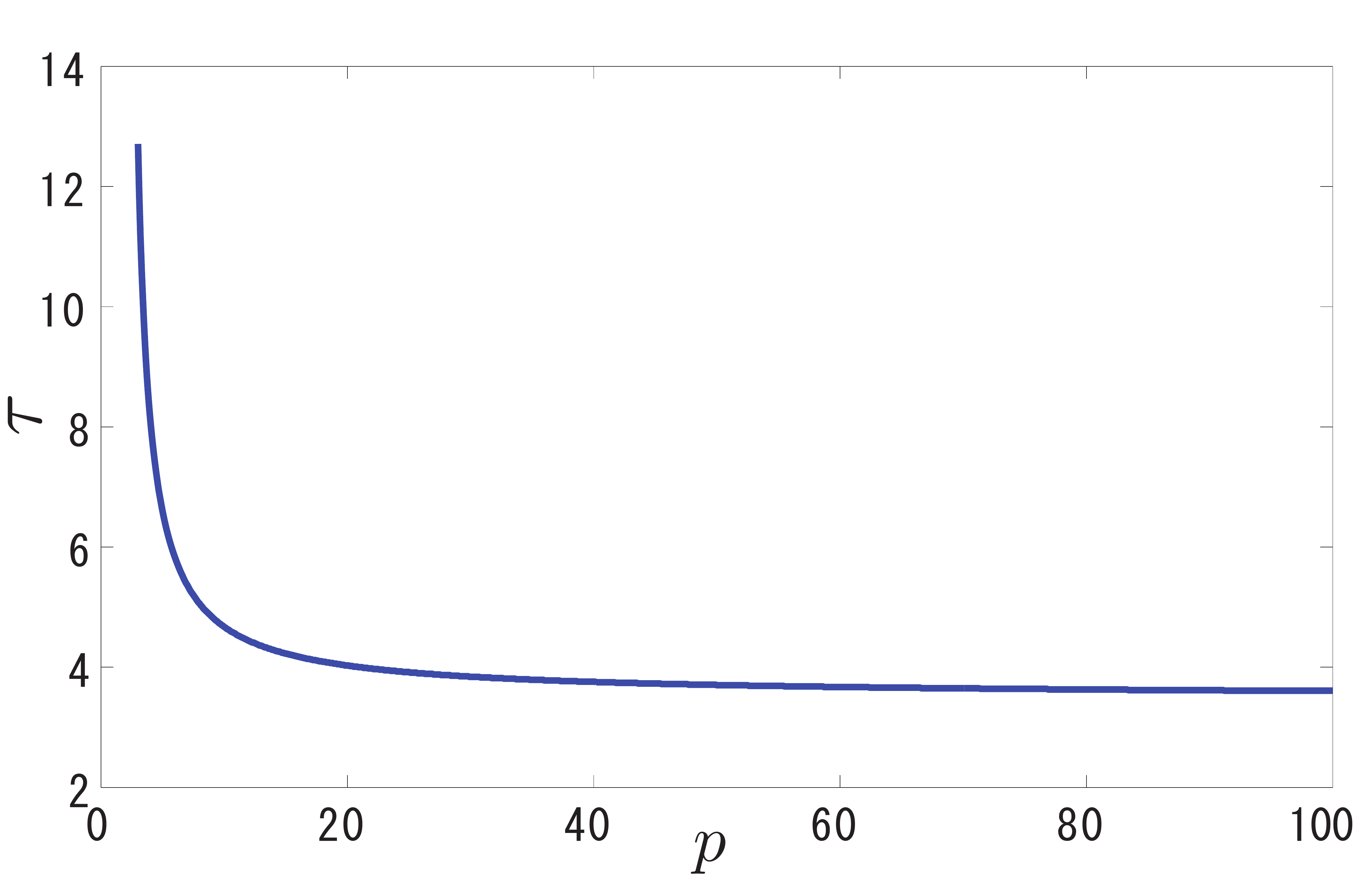}\\[2pt]
~\,{\Large (b)}
\end{center}
\end{minipage}
\begin{minipage}{\hsize}
\begin{center}
  \includegraphics[height=50mm]{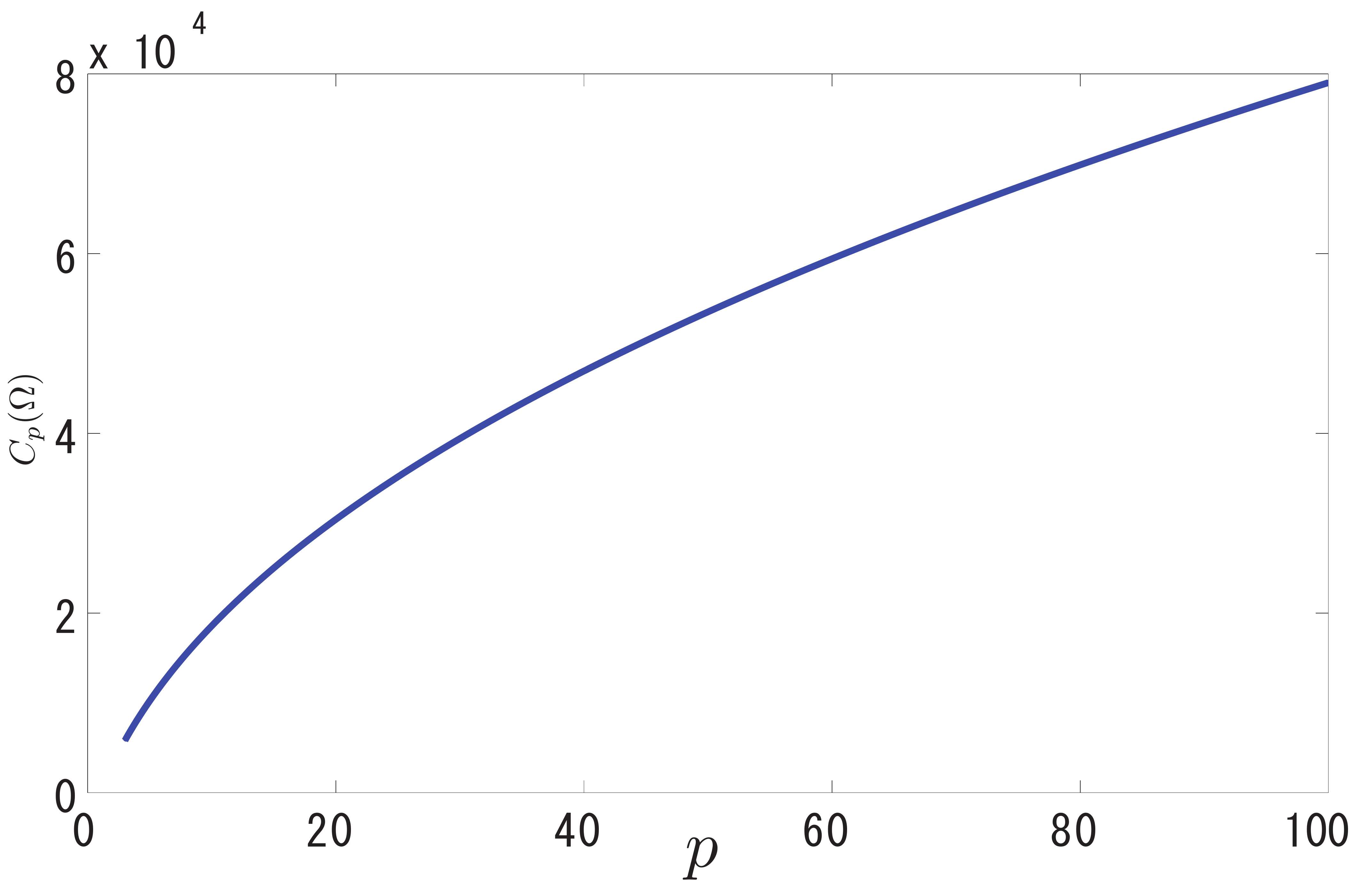}\\
~{\Large (c)}
\end{center}
\end{minipage}\\
\caption{Same as Fig.~\ref{hole/graph},
but in the case of the domain $\Omega$ in Fig.~\ref{ufo} (a).
}
\label{ufo/graph}
\end{figure}

\newpage

\section{Conclusion}
We proposed the method for estimating the operator norm $A_{q}\left(\Omega\right)$ (defined in (\ref{intro/aq})) of the extension operator constructed by Stein \cite{stein1970book}.
The concrete bounds for the operator norm leads to estimate the embedding constant $C_{p}(\Omega)$ from $W^{1,q}(\Omega)$ to $L^{p}(\Omega)$ defined in (\ref{purpose}).
Here, $\Omega$ is only assumed to be a domain with minimally smooth boundary.

In addition, we presented some estimation results of the embedding constants.
All estimation results are mathematically guaranteed with verified numerical computation, while the derived constants may not be sharp because of some over-estimations.

\appendix
\section{The best constant in the classical Sobolev inequality}\label{apA}
The following theorem gives the best constant in the classical Sobolev inequality.
\begin{theo}[T.~Aubin \cite{aubin1976} and G.~Talenti \cite{talenti1976}]\label{talentitheo}
Let $u$ be any function in $H^{1}\left(\mathbb{R}^{n}\right)\ (n=2,3,\cdots)$.
Moreover, let $q$ be any real number such that $1<q<n$, and set $p=nq/\left(n-q\right)$.
Then,
\begin{align*}
\left\|u\right\|_{L^{p}\left(\mathbb{R}^{n}\right)}\leq T_{p}\left\|\nabla u\right\|_{L^{q}\left(\mathbb{R}^{n}\right)}
\end{align*}
holds for
\begin{align*}
T_{p}=\pi^{-\frac{1}{2}}n^{-\frac{1}{q}}\left(\frac{q-1}{n-q}\right)^{1-\frac{1}{q}}\left\{\frac{\Gamma\left(1+\frac{n}{2}\right)\Gamma\left(n\right)}{\Gamma\left(\frac{n}{q}\right)\Gamma\left(1+n-\frac{n}{q}\right)}\right\}^{\frac{1}{n}}
\end{align*}
with the Gamma function $\Gamma$.
\end{theo}
\section{Lemmas for proving Lemma \ref{sp/lemma} and Theorem \ref{main}}\label{apC}
The following two lemmas were used in the proof of Lemma \ref{sp/lemma} and Theorem \ref{main}.
\begin{lem}[G.H.~Hardy, et al. \cite{hardy1952book}]\label{hardy}
Let $p\in \mathbb{N}$ and let $r>0$.
Suppose that a function $f:\mathbb{R}\rightarrow \mathbb{R}$ satisfies $f\left(x\right)\geq 0,\ \forall x\in \mathbb{R}$.
Then, it follows that
\begin{align*}
\left(\int_{0}^{\infty}\left(\int_{0}^{x}f\left(y\right)dy\right)^{p}x^{-r-1}dx\right)^{1/p}&\displaystyle \leq\frac{p}{r}\left(\int_{0}^{\infty}\left(yf\left(y\right)\right)^{p}y^{-r-1}dy\right)^{1/p},
\shortintertext{and}
\left(\int_{0}^{\infty}\left(\int_{x}^{\infty}f\left(y\right)dy\right)^{p}x^{r-1}dx\right)^{1/p}&\displaystyle \leq\frac{p}{r}\left(\int_{0}^{\infty}\left(yf\left(y\right)\right)^{p}y^{r-1}dy\right)^{1/p}.
\end{align*}
\end{lem}
\begin{lem}\label{fact}
Let $S\subseteq \mathbb{R}^{n}$ and $p\in[1,\infty)$.
Moreover, let $\left\{a_{i}\left(x\right)\right\}_{i\in \mathbb{N}}\subset L^{p}\left(S\right)$ satisfy that at most $N$ of $a_{i}\left(x\right)$ are not zero for each $x$.
Then, it follows that
\begin{align*}
\left(\int_{S}\left|\sum_{i\in \mathbb{N}}a_{i}\left(x\right)\right|^{p}dx\right)^{\frac{1}{p}}\leq N^{1-\frac{1}{p}}\left(\sum_{i\in \mathbb{N}}\int_{S}\left|a_{i}\left(x\right)\right|^{p}dx\right)^{\frac{1}{p}}.
\end{align*}
\end{lem}
\begin{proof}
This lemma follows from the following inequality:
\begin{align*}
\displaystyle \left|\sum_{i\in \mathbb{N}}a_{i}\left(x\right)\right|^{p}\leq N^{p-1}\sum_{i\in \mathbb{N}}\left|a_{i}\left(x\right)\right|^{p}.
\end{align*}
\end{proof}
\section{The embedding constant from $H^{1}\left(\Omega\right)$ to $L^{p}(\Omega)$}\label{apB}
Corollary \ref{coroap}, which comes from Theorem \ref{main}, gives a concrete estimation of the embedding constant from $H^{1}\left(\Omega\right)$ to $L^{p}(\Omega)$ under the suitable assumptions for $\Omega$ and $p$.
\begin{coro}\label{coroap}
Let $n\in\left\{2,3,\cdots\right\}$ and let $p$ be given, s.t., $p\in(n/(n-1),2n/(n-2))$ if $n\geq 3$ and $p\in(n/(n-1),\infty)$ if $n=2$.
Let $T_{p}$ be a constant in the classical Sobolev inequality, i.e.,
$\left\|u\right\|_{L^{p}\left(\mathbb{R}^{n}\right)}\leq T_{p}\left\|\nabla u\right\|_{L^{q}\left(\mathbb{R}^{n}\right)}$
for all
$u\in W^{1,q}(\mathbb{R}^{n})$,
where $q=np/(n+p)$.
Moreover, let $\Omega\subset \mathbb{R}^{n}$ be a bounded domain with minimally smooth boundary.
Then,
\begin{align*}
\left\|u\right\|_{L^{p}(\Omega)}\leq C_{p}'\left(\Omega\right)\left\|u\right\|_{H^{1}(\Omega)},\ \forall u\in H^{1}(\Omega),
\end{align*}
holds for
\begin{align*}
C_{p}'\left(\Omega\right)=2^{1/2}\left|\Omega\right|^{\frac{2-q}{2q}}T_{p}A_{q}\left(\Omega\right).
\end{align*}
Here, $\left\|\cdot\right\|_{W^{1,q}(\Omega)}$ denotes the $\sigma$-weighted $W^{1,q}$ norm $(\ref{sigmanorm})$ for given $\sigma>0$, and $A_{q}\left(\Omega\right)$ is the upper bound for the operator norm derived by Theorem \ref{main} with $\gamma=\sigma^{1/2}$.
\end{coro}
\begin{proof}
Let $u\in H^{1}\left(\Omega\right)$.
From the same discussion in (\ref{proof/coro}), it follows that
\begin{align}
\left\|u\right\|_{L^{p}\left(\Omega\right)}\leq T_{p}A_{q}\left(\Omega\right)\left(\left\|\nabla u\right\|_{L^{q}\left(\Omega\right)}+\sigma^{1/2}\left\|u\right\|_{L^{q}\left(\Omega\right)}\right).\label{embedding/theo/1}
\end{align}
Since $q\in(1,2)$ holds when $p\in(n/(n-1),2n(n-2))$, H\"{o}lder's inequality gives
\begin{align}
\left\|\nabla u\right\|_{L^{q}\left(\Omega\right)}^{q}&\leq\left(\int_{\Omega}\left|\nabla u\left(x\right)\right|^{q\cdot\frac{2}{q}}dx\right)^{\frac{q}{2}}\left(\int_{\Omega}\left|1\right|^{\frac{2}{2-q}}dx\right)^{\frac{2-q}{2}}\nonumber\\
&=\left|\Omega\right|^{\frac{2-q}{2}}\left(\int_{\Omega}\left|\nabla u\left(x\right)\right|^{2}dx\right)^{\frac{q}{2}},\nonumber
\end{align}
where $\left|\Omega\right|$ is the measure of $\Omega$.
Therefore,
\begin{align}
\left\|\nabla u\right\|_{L^{q}\left(\Omega\right)}\leq\left|\Omega\right|^{\frac{1}{p}}\left\|\nabla u\right\|_{L^{2}\left(\Omega\right)}.\label{embedding/theo/2}
\end{align}
In the same manner, we have
\begin{align}
\left\|u\right\|_{L^{q}\left(\Omega\right)}\leq\left|\Omega\right|^{\frac{1}{p}}\left\|u\right\|_{L^{2}\left(\Omega\right)}.\label{embedding/theo/3}
\end{align}
From (\ref{embedding/theo/1}), (\ref{embedding/theo/2}), and (\ref{embedding/theo/3}),
\begin{align*}
\left\|u\right\|_{L^{p}\left(\Omega\right)}&\leq\left|\Omega\right|^{\frac{2-q}{2q}}T_{p}A_{q}\left(\Omega\right)\left(\left\|\nabla u\right\|_{L^{2}\left(\Omega\right)}+\sigma^{1/2}\left\|u\right\|_{L^{2}\left(\Omega\right)}\right)\\
&\leq 2^{1/2}\left|\Omega\right|^{\frac{2-q}{2q}}T_{p}A_{q}\left(\Omega\right)\left\|u\right\|_{H^{1}(\Omega)}.
\end{align*}
\end{proof}

\bibliographystyle{amsplain} 
\bibliography{ref.bib}

\providecommand{\bysame}{\leavevmode\hbox to3em{\hrulefill}\thinspace}
\providecommand{\MR}{\relax\ifhmode\unskip\space\fi MR }
\providecommand{\MRhref}[2]{%
  \href{http://www.ams.org/mathscinet-getitem?mr=#1}{#2}
}
\providecommand{\href}[2]{#2}
\begin{thebibliography}{10}

\bibitem{adams2003book}
Robert~A Adams and John~JF Fournier, \emph{Sobolev spaces}, second ed., Pure
  and applied mathematics, Academic Press, New York, 2003.

\bibitem{aubin1976}
Thierry Aubin, \emph{Probl\`emes isop\'erim\'etriques et espaces de {Sobolev}},
  Journal of Differential Geometry \textbf{11} (1976), no.~4, 573--598.

\bibitem{aubin1982book}
\bysame, \emph{Nonlinear analysis on manifold. {Monge-Amp\`ere} equations},
  Springer, New York, 1982.

\bibitem{beckner1993}
William Beckner, \emph{Sharp {Sobolev} inequalities on the sphere and the
  {Moser--Trudinger} inequality}, Annals of Mathematics \textbf{138} (1993),
  no.~1, 213--242.

\bibitem{brezis2011book}
Haim Brezis, Haim Br{\'e}zis, and Ha{\"\i}m Brezis, \emph{Functional analysis,
  {Sobolev} spaces and partial differential equations}, Springer, New York,
  2011.

\bibitem{brezis1983}
Haim Brezis and Louis Nirenberg, \emph{Positive solutions of nonlinear elliptic
  equations involving critical {Sobolev} exponents}, Communications on Pure and
  Applied Mathematics \textbf{36} (1983), no.~4, 437--477.

\bibitem{calderon1961}
Alberto~P Calder{\'o}n, \emph{Lebesgue spaces of differentiable functions and
  distributions}, Proc. Sympos. Pure Math, vol.~4, 1961, pp.~33--49.

\bibitem{edmunds1987book}
David~Eric Edmunds and WD~Evans, \emph{Spectral theory and differential
  operators}, Oxford University Press, New York, 1987.

\bibitem{fraenkel1979}
LE~Fraenkel, \emph{On regularized distance and related functions}, Proceedings
  of the Royal Society of Edinburgh: Section A Mathematics \textbf{83} (1979),
  no.~1-2, 115--122.

\bibitem{hardy1952book}
Godfrey~Harold Hardy, John~Edensor Littlewood, and George P{\'o}lya,
  \emph{Inequalities}, Cambridge university press, Cambridge, 1952.

\bibitem{hestenes1941}
Magnus~Rudolph Hestenes et~al., \emph{Extension of the range of a
  differentiable function}, Duke Mathematical Journal \textbf{8} (1941), no.~1,
  183--192.

\bibitem{lieb1983}
Elliott~H Lieb, \emph{Sharp constants in the {Hardy-Littlewood-Sobolev} and
  related inequalities}, Annals of Mathematics (1983), 349--374.

\bibitem{nakao2001numerical}
Mitsuhiro~T Nakao, \emph{Numerical verification methods for solutions of
  ordinary and partial differential equations}, Numerical Functional Analysis
  and Optimization \textbf{22} (2001), no.~3-4, 321--356.

\bibitem{plum2001computer}
Michael Plum, \emph{Computer-assisted enclosure methods for elliptic
  differential equations}, Linear Algebra and its Applications \textbf{324}
  (2001), no.~1, 147--187.

\bibitem{plum2008}
\bysame, \emph{Existence and multiplicity proofs for semilinear elliptic
  boundary value problems by computer assistance}, Jahresbericht der Deutschen
  Mathematiker Vereinigung \textbf{110} (2008), no.~1, 19--54.

\bibitem{rump1999book}
{S.M.} Rump, \emph{{INTLAB - INTerval LABoratory}}, {Developments~in~Reliable
  Computing} (Tibor Csendes, ed.), Kluwer Academic Publishers, Dordrecht, 1999,
  \url{http://www.ti3.tuhh.de/rump/}, pp.~77--104.

\bibitem{stein1970book}
EM~Stein, \emph{Singular integrals and differentiability of functions},
  Princeton University Press, New Jersey, 1970.

\bibitem{takayasu2013verified}
Akitoshi Takayasu, Xuefeng Liu, and Shin'ichi Oishi, \emph{Verified
  computations to semilinear elliptic boundary value problems on arbitrary
  polygonal domains}, Nonlinear Theory and Its Applications, IEICE \textbf{4}
  (2013), no.~1, 34--61.

\bibitem{talenti1976}
Giorgio Talenti, \emph{Best constant in {Sobolev} inequality}, Annali di
  Matematica pura ed Applicata \textbf{110} (1976), no.~1, 353--372.

\bibitem{watanabe2009symmetrization}
Kohtaro Watanabe, Yoshinori Kametaka, Atsushi Nagai, Hiroyuki Yamagishi, and
  Kazuo Takemura, \emph{Symmetrization of functions and the best constant of
  1-dim sobolev inequality}, Journal of Inequalities and Applications
  \textbf{2009} (2009), no.~1, 874631.

\bibitem{whitney1934}
Hassler Whitney, \emph{Analytic extensions of differentiable functions defined
  in closed sets}, Transactions of the American Mathematical Society
  \textbf{36} (1934), no.~1, 63--89.

\end{thebibliography}

\end{document}